\documentclass[11pt]{article} 

\usepackage[utf8]{inputenc} 
\usepackage{latexsym,amsmath,amssymb,textcomp,amsthm}

\usepackage{geometry} 
\usepackage{tikz}
\usepackage{graphicx} 

\usepackage{calrsfs}
\usepackage[english]{babel}
\usepackage{booktabs} 
\usepackage{array} 
\usepackage{paralist} 
\usepackage{verbatim} 
\usepackage{subfig} 
\usepackage{pgf}
\usepackage{fancyhdr} 
\usepackage{hyperref}
\usepackage{sectsty}
\allsectionsfont{\sffamily\mdseries\upshape} 
\usepackage{amsmath,amsfonts,amssymb}

\usepackage[nottoc,notlof,notlot]{tocbibind} 
\usepackage[titles,subfigure]{tocloft} 

\usetikzlibrary{arrows,automata}

\newtheorem{prop}{Proposition}[subsection]
\newtheorem{lem}[prop]{Lemma}

\newtheorem{theo}{Theorem}
\newtheorem*{theo2}{Theorem}
\newtheorem{defin}{Definition}
\newtheorem{rmq}{Remark}

\newcommand{\T}[1]{\, ^t\! #1}

\newcommand{\cg}{[\kern-0.15em [}
\newcommand{\cd}{]\kern-0.15em]}

\newcommand{\E}{\mathbb{E}}
\newcommand{\Prob}{\mathbb{P}}

\newcommand{\R}{\mathbf{R}}

\newcommand{\N}{\mathbf{N}}
\newcommand{\Z}{\mathbf{Z}}

\newcommand{\dd}{\mathrm{d}}

\title{Monotonicity and phase transition for the VRJP and the ERRW}
\author{R\'emy Poudevigne--Auboiron}
\date{} 

\begin{document}
\maketitle

\begin{abstract}
The vertex-reinforced jump process (VRJP), introduced by Davis and Volkov in \cite{VRJPIntro}, is a continuous-time process that tends to come-back to already visited vertices. It is closely linked to the edge-reinforced random walk (ERRW) introduced by Coppersmith and Diaconis in 1986 (\cite{linreinforced}) which is more likely to cross edges it has already crossed. On $\Z^d$ for $d\geq 3$, both models where shown to be recurrent for small enough initial weights (\cite{ERRWVRJP},\cite{renfexp}) and transient for large enough initial weights (\cite{ERRWTrans},\cite{ERRWVRJP}). We show through a coupling of the VRJP for different weights that the VRJP (and the ERRW) exhibits some monotonicity. In particular, we show that increasing the initial weights of the VRJP and the ERRW makes them more transient which means that the recurrence/transience phase transition is necessarily unique. Furthermore, by making the weights go to infinity, we show that the recurrence of the ERRW and the VRJP is implied by the recurrence of a random walk in deterministic electrical network.
\end{abstract}
\section{Introduction and results}
\subsection{Introduction}

The edge-reinforced random walk (ERRW) was first introduced by Coppersmith and Diaconis in 1986 $\cite{linreinforced}$. In this model, the more the walk crosses an edge, the likelier it is to cross it again in the future. This model was shown to be a random walk in random reversible environments $(\cite{ERRWMixture},\cite{ERRWMixture2})$. This representation lead to several results on this model, first recurrence and transience on trees depending on the reinforcement $\cite{Pemantle}$ then recurrence on the ladder $\cite{ERRWLadder}$ and $\Z\times G$ $\cite{ERRWLadder2}$ for large enough reinforcement and on a modification of $\Z^2$ for large enough reinforcement $\cite{lerrwrec2d}$. It was then shown by two different techniques that the ERRW on $\Z^d$ is recurrent for large enough reinforcement (in \cite{renfexp} by Angel, Crawford and Kozma and in \cite{ERRWVRJP} by Sabot and Tarr\`es). The technique used in $\cite{ERRWVRJP}$ was based on a link between the ERRW, the vertex-reinforced jump process (VRJP, introduced by Davis and Volkov in \cite{VRJPIntro}) and the super-symmetric hyperbolic sigma model (introduced in the context of random band matrices in \cite{SUSY1},\cite{SUSY2} by Zirnbauer, Disertori and Spencer). This relation led to several other results for both the ERRW and the VRJP: the transience and a CLT in dimension 3 and higher for small enough reinforcements (\cite{ERRWTrans},\cite{ERRWVRJP},\cite{VRJPNu}), a $0-1$ law for recurrence on $\Z^d$ \cite{VRJPNu} and the recurrence in dimension 2 (\cite{VRJPNu},\cite{lerrwrec2d},\cite{RecVRJPZ2}). This means that on the one hand, for $d\in \{1,2\}$ the ERRW and the VRJP are recurrent for any reinforcement. On the other hand, for $d\geq 3$ both the ERRW and the VRJP are recurrent for large enough reinforcement and transient for small enough reinforcements. We know that in-between, the VRJP and the ERRW are recurrent or transient but it was not known whether there is a unique phase transition. In this paper we show that we can couple the VRJP for different weights (more precisely, we couple the $\beta$-field associated to the VRJP that was introduced in \cite{VRJPNu0}). This coupling leads to a monotonicity for the VRJP similar to the Rayleigh monotonicity for electrical networks. This gives us the uniqueness of the recurrence/transience phase transition for the VRJP and the ERRW in dimension $3$ and higher. This monotonicity can also be used to show that the VRJP and the ERRW with constant weights are recurrent on recurrent graphs by seeing random walks in electrical networks as VRJPs with infinite weights.

\subsection{Statement of the results}
Let $\mathcal{G}=(V,E)$ be a locally finite, non-directed graph. To every edge $e\in E$ we associate a positive weight $a_e$. Let $x_0\in V$ be a vertex of $\mathcal{G}$. The edge-reinforced random walk $Y$ starting from $x_0$ is the random process which takes its values in $V$ defined by: 
\[
Y_0=x_0 \text{ a.s, and}
\]
\[
\Prob\left(Y_{n+1}=y|Y_0,\dots,Y_n\right)=1_{y\sim Y_n} \frac{a_{\{Y_n,y\}}+Z_n(\{Y_n,y\})}{\sum\limits_{z\sim Y_n}a_{\{Y_n,z\}}+Z_n(\{Y_n,z\})},
\]
where the random variables $(Z_n)_{n\in \N}$ are defined by:
\[
\forall e\in E,\ Z_n(e)=\sum\limits_{i=0}^{n-1}1_{\{Y_i,Y_{i+1}\}=e}.
\]
If the graph is $\Z^d$, this process can exhibit different behaviours depending on the initial weights. For small enough initial weights it is recurrent.
\begin{theo2}[Theorem 1 of \cite{renfexp} and corollary 2 of \cite{ERRWVRJP}]
For any $K$ there exists $a_0>0$ such that if $\mathcal{G}$ is a graph with all degrees bounded by $K$, then the linearly edge reinforced random walk on $\mathcal{G}$ with initial weights $a\in(0,a_0)$ is positive recurrent.
\end{theo2}
For large enough initial weights, the process is transient.
\begin{theo2}[Theorem 1 of \cite{ERRWTrans}]
On $\Z^d$, $d\geq 3$, there exists $a_c(d)>0$ such that, if $a_e>a_c(d)$ for all $e\in E$, then the ERRW with weights $(a_e)_{e\in E}$ is transient a.s.
\end{theo2}
Note that the previous two theorems use results or ideas of \cite{SUSY2} and \cite{SUSY3}. The ERRW is linked to an other random process, the vertex-reinforced jump process (VRJP). The VRJP on a locally finite graph $\mathcal{G}=(V,E)$ is the continuous-time process $(\tilde{Y}_t)_{t\in\R^+}$ that starts at some vertex $x_0$ and that, conditionally on the past at time $t$, if $\tilde{Y}_t=x$, jumps to a neighbour $y$ of $x$ at rate 
\[
W_{\{x,y\}}\ell_x(t),
\]
where
\[
\ell_x(t):=\int\limits_{0}^t 1_{\tilde{Y}_s=x}\dd s.
\]
The following link between the ERRW and the VRJP has been shown in $\cite{ERRWVRJP}$.
\begin{theo2}[Theorem 1 of \cite{ERRWVRJP}]
The ERRW with weights $(a_e)_{e\in E}$ is equal in law to the discrete time process associated with a VRJP in random independent weights $W_e \sim \text{Gamma}(a_e,1)$.
\end{theo2}

In this article we show, through a coupling, that the VRJP has a property similar to Rayleigh's monotonicity for electrical network. This leads to several results for recurrence and transience. First, we show that the probability that the walk is recurrent is decreasing in the parameters of the VRJP. This is a corollary of our main theorem that will be stated at the end because it is technical and needs a few additional definitions.

\begin{theo}\label{theo:TransCroiss}
Let $\mathcal{G}=(V,E)$ be an infinite, non-directed, connected graph without loops or multiple edges and $0\in V$ a vertex in this graph. Let $(W_e^-)_{e\in E}$ and $(W_e^+)_{e\in E}$ be two families of positive weights such that for any $e\in E$, $0<W_e^-\leq W_e^+$. The probability that the VRJP with initial weights $W^-$ is recurrent is greater or equal than the probability that the VRJP with initial weights $W^+ $ is recurrent.
\end{theo}

It was already proved that the VRJP on $\Z^d$ with constant weights or weights invariant by translation is recurrent with probability $0$ or $1$ in \cite{VRJPNu}. In addition to our theorem this means that the VRJP and the ERRW are recurrent for small enough weights and then transient for larger weights. This means that the VRJP and the ERRW exhibit a phase transition for recurrence/transience on $\Z^d$ when all the edges have the same weight.

\begin{theo}\label{theo2}
Set $d\geq 3$ there exists $w_d\in (0,\infty)$ such that the VRJP on $\Z^d$ with initial weight $w\in (0,\infty)$ is recurrent if $w<w_d$ and transient if $w>w_d$.
\end{theo}

\begin{theo}\label{theo3}
Set $d\geq 3$ there exists $a_d\in (0,\infty)$ such that the ERRW on $\Z^d$ with initial weight $a\in (0,\infty)$ is recurrent if $a<a_d$ and transient if $a>a_d$.
\end{theo}

The link between the VRJP and electrical network goes beyond this monotonicity property. The following theorem shows that recurrence of electrical networks, VRJP and ERRW are also closely linked.
\begin{theo}\label{theo:CondEff}
Let $\mathcal{G}=(V,E)$ be an infinite, locally finite graph and $x_0\in V$ a vertex. Let $(W_e)_{e\in E}$ be a family of positive weights. If the random walk on $\mathcal{G}$ starting at $x_0$ with deterministic conductances $(c_e)_{e\in E}=(W_e)_{e\in E}$ is recurrent then so are the ERRW and the VRJP starting at $x_0$ and with initial weights $(W_e)_{e\in E}$.
\end{theo}

To state our technical main theorem, we need some extra definition and results related to the VRJP and the ERRW. First we need to introduce the $\beta$-field (introduced in \cite{VRJPNu0} by Tarr\`es, Sabot and Zeng), a random vector defined for weighted graphs.
\begin{defin}
Let $n$ be an integer, $(\eta_i)_{1\leq i \leq n}$ a family of non-negative parameters and $W\in M_{n}(\R)$ a symmetric matrix with non-negative coefficients. Let $1_n\in\R^n$ be the vector $(1,\dots,1)$. The measure $\nu_n^{W,\eta}$ on $(0,\infty)^n$ is defined by the following density:
\[
\nu_n^{W,\eta}\left(\dd\beta_1\dots\dd\beta_n\right):=e^{-\frac{1}{2}\left(1_n H_{\beta}\T{1_n}\ +\ \eta H_{\beta}^{-1}\T{\eta}\ -\ 2 \sum\limits_{1\leq i\leq n} \eta_i\right)}
\frac{1}{\sqrt{\text{det}(H_{\beta})}}1_{H_{\beta}>0}\dd\beta_1\dots\dd\beta_n,
\]
where $\forall i,j\in[\![1,n]\!]$,
\[
\begin{aligned}
H_{\beta}(i,i)=&2\beta_i-W(i,i),\\
H_{\beta}(i,j)=&-W(i,j) \text{ if } i\not= j
\end{aligned}
\]
and $H_{\beta}>0$ means that $H_{\beta}$ is positive definite. \\
This family of measures is actually a family of probability measures, as was proved in \cite{VRJPNu0}.\\
We call $\tilde{\nu}_n^{W,\eta}$ the distribution of $H_{\beta}$ when $(\beta_i)_{1\leq i \leq n}$ is distributed according to $\nu_n^{W,\eta}$.
\end{defin}
The link between the $\beta$-field and the VRJP is not obvious at first glance. It was shown in \cite{VRJPNu0} (based on previous results in \cite{ERRWVRJP}) that the VRJP with weights $W$ can be seen as a random walk in a random electrical network whose conductances are given by the weights $W$ and the $\beta$-field. More precisely:
\begin{theo2}[Theorem 3 of \cite{VRJPNu0}]
Let $\mathcal{G}=(V,E)$ be a non-directed graph and $(W_e)_{e\in E}$ weights on the edges. Let $H_{\beta}$ be distributed according to $\tilde{\nu}^{W,0}_{|V|}$ and let $G_{\beta}$ be the inverse of $H_{\beta}$. For any $x_0\in V$ the discrete path of the VRJP (the sequence of vertices at each successive jump) on $\mathcal{G}$ with weights $W$, starting at $x_0$, is a random walk in random electrical network where the conductances $(c_e)_{e\in E}$ are given by:
\[
c_{\{x,y\}}=W_{\{x,y\}}G_{\beta}(x_0,x)G_{\beta}(x_0,y).
\] 
\end{theo2}
The reason we look at the $\beta$-field instead of the conductances is that the $\beta$-field has several interesting properties. First, the $\beta$-field does not depend on the starting point of the VRJP. Its Laplace transform has a simple expression and it is 1-dependent. But most importantly, the family of laws $\nu^{W,\eta}$ is stable by taking marginals or conditional distributions (lemma 5 of \cite{VRJPNu} and independently in \cite{Letac}). More precisely:
\begin{prop}\label{lem:VRJPNu0}
Let $n_1,n_2$ be two integers, and $n:=n_1+n_2$. Let $W\in M_n(\R)$ be a symmetric matrix with non-negative coefficients and $(\eta_i)_{i\in[\![1,n_1+n_2]\!]}$ a family of non-negative coefficients. Let $(\beta_i)_{i\in [\![1,n_1+n_2]\!]}$ be random variables with a $\nu_{n}^{W,\eta}$ distribution and $H_{\beta}\in M_n(\R)$ the matrix defined by: 
\[
\forall i,j\in[\![1,n]\!], H_{\beta}(i,j):= \left\{ \begin{matrix}
2\beta_i-W(i,i) \text{ if } i= j,\\
-W(i,j) \text{ if } i\not= j.
\end{matrix}
\right.
\]
We make the following bloc decomposition:
\[
W=\left(\begin{matrix} W^{11} & W^{12} \\ W^{21} & W^{22}\end{matrix}\right),
H_{\beta}=\left(\begin{matrix} H_{\beta}^{11} & H_{\beta}^{12} \\ H_{\beta}^{21} & H_{\beta}^{22}\end{matrix}\right)
\text{ and }
\eta=\left(\begin{matrix}\eta^1 \\ \eta^2 \end{matrix}\right),
\]
where $W^{11},H_{\beta}^{11}\in M_{n_1}(\R)$, $W^{12},H_{\beta}^{12}\in M_{n_1,n_2}(\R)$, $W^{21},H_{\beta}^{21}\in M_{n_2,n_1}(\R)$, $W^{22},H_{\beta}^{22}\in M_{n_2}(\R)$, $\eta^1\in \R^{n_1}$ and $\eta^2\in \R^{n_2}$. Then the family $(\beta_i)_{1\leq i \leq n_1}$ is distributed according to $\nu_{n_1}^{W^{11},\hat{\eta}}$ where
\[
\hat{\eta}\in\R^{n_1} \text{ and } \forall i \in [\![1,n_1]\!], \  \hat{\eta}_i:=\eta_i + \sum\limits_{k=1}^{n_2} W^{12}(i,k).
\] 
Conditionally on $(\beta_i)_{1\leq i \leq n_1}$, the family $(\beta_i)_{n_1+1\leq i \leq n_1+n_2}$ is distributed according to $\nu_{n_2}^{\check{W},\check{\eta}}$ where
\[
\check{W}=W^{22}+W^{21}\left(H_{\beta}^{11}\right)^{-1}W^{12},
\]
and
\[
\check{\eta}\in\R^{n_2} \text{ and } \check{\eta}=\eta^2 + W^{21}\left(H_{\beta}^{11}\right)^{-1} \eta^1.
\]
\end{prop}

\begin{defin}
Let $n$ be an integer and let $H\in M_n(\R)$ be a symmetric matrix. We say that two integers $1\leq i,j\leq n$ are $H$-connected if there exists a finite sequence $(k_1,\dots,k_m)$ such that $k_1=i,k_m=j$ and for all $1\leq a \leq m-1$, $H(k_a,k_{a+1})\not=0$. 
\end{defin}

We can now state our (technical) main theorem which gives a coupling between VRJPs of different weights and a simpler corollary that is the equivalent of Rayleigh monotonicity for the VRJP.
\begin{theo}\label{maintheo}
Set an integer $n\in\N$. Let $W\in M_n(\R)$ be a symmetric matrix with non-negative off diagonal coefficients and null diagonal coefficients. Let $W^1,W^2\in M_{n,1}(\R)$ be two matrices with non-negative coefficients and let $W^{3}\in M_{n,1}(\R)$ be the matrix defined by $W^{3}:=W^{1}+W^{2} $. Let $w^-,w^+\in [0,\infty)$ be two positive real with $w^-<w^+$. We define the matrices $W^-,W^+$ and $W^{\infty}$ by:
\[
W^-:=\left(\begin{matrix}W & W^{1} & W^2\\ \T{W^{1}} & 0 & w^- \\ \T{W^{2}} & w^- & 0 \end{matrix}\right),
W^+:=\left(\begin{matrix}W & W^{1} & W^2\\ \T{W^{1}} & 0 & w^+ \\ \T{W^{2}} & w^+ & 0 \end{matrix}\right) \text{ and }
W^{\infty}:=\left(\begin{matrix}W & W^{3} \\ \T{W^{3}} & 0 \end{matrix}\right).
\]  
If $n=0$, we just have:
\[
W^-:=\left(\begin{matrix}  0 & w^- \\ w^- & 0 \end{matrix}\right) ,
W^+:=\left(\begin{matrix}  0 & w^+ \\ w^+ & 0 \end{matrix}\right) \text{ and }
W^{\infty}:=\left(\begin{matrix} 0 \end{matrix}\right).
\]
For any vector $X\in\R^{n+2}$ we define the vector $\overline{X}\in\R^{n+1}$ by:
\[
\begin{aligned}
\forall i\in[\![1,n]\!],\ \overline{X}_i:=X_i \text{ and }\\
\overline{X}_{n+1}:=X_{n+1}+X_{n+2}.
\end{aligned}
\]
For any vector $X^1\in[0,\infty)^{n+2}$ there exists random matrices $H^-,H^+$ and $H^{\infty}$ (with inverse $G^-,G^+$ and $G^{\infty}$ respectively) that are distributed according to $\tilde{\nu}_{n+2}^{W^-,0},\tilde{\nu}_{n+2}^{W^+,0}$ and $\tilde{\nu}_{n+1}^{W^{\infty},0}$ respectively such that
\[
\T{X}^1 G^- X^1=\T{X}^1 G^+ X^1=\T{\overline{X^1}} G^{\infty} \overline{X^1} \text{ almost surely,}
\]
for all $i\in [\![1,n]\!]$, $H^-(i,i)=H^+(i,i)=H^{\infty}(i,i)$ and for any vector $X^2\in[0,\infty)^{n+2}$ we have:
\[
\begin{aligned}
\E\left(\T{X}^1 G^+ X^2|H^{\infty}\right)=&\T{\overline{X^1}} G^{\infty} \overline{X^2},\text{ and }\\
\E\left(\T{X}^1 G^- X^2|H^{+}\right)=&\T{X}^1 G^{+} X^2 \text{ if } n+1 \text { and } n+2 \text{ are } H^-\text{-connected.} 
\end{aligned}
\]
\end{theo}
It was already known that a special case of this theorem was true: the martingale property between $H^+$ and $H^{\infty}$ under specific assumptions (the martingale property for $\psi$ in \cite{VRJPNu}). However, the link between $H^+$ and $H^-$ was not known.
\begin{theo}\label{maincor}
Let $n\geq 2$ be an integer, let $W^-,W^+\in M_n(\R)$ be two symmetric matrices with null diagonal coefficients and non-negative off-diagonal coefficients such that for any $i,j\in[\![1,n]\!]$, $W^-(i,j)\leq W^+(i,j)$ and $i$ and $j$ are $W^-$-connected. Let $H^-$ and $H^+$ be two matrices distributed according to $\tilde{\nu}_n^{W^-,0}$ and $\tilde{\nu}_n^{W^+,0}$ respectively, and let their inverse be $G^-$ and $G^+$ respectively. For any convex function $f$, any integer $i\in[\![1,n]\!]$ and any deterministic vector $X\in[0,\infty)^n$:
\[
\E\left(f\left(\frac{\sum\limits_{j=1}^n X_i G^-(i,j)}{G^-(i,i)}\right)\right)\geq \E\left(f\left(\frac{\sum\limits_{j=1}^n X_i G^+(i,j)}{G^+(i,i)}\right)\right).
\]
\end{theo}
For a specific choice of $X$ and a specific choice of $i$, the random variable $\frac{\sum\limits_{j=1}^n X_i G(i,j)}{G(i,i)}$ is equal to the random variable $\psi$ defined in \cite{VRJPNu} (to be more precise, it is equal to an approximation of $\psi$ on finite graphs). This random variable $\psi$ is closely linked to the recurrence of the graph (it is equal to $0$ iff the VRJP is recurrent). By using this theorem for $\psi$ (to be more precise, on an approximation of $\psi$ on finite graphs), it is then possible to deduce the uniqueness of the phase transition between recurrence and transience for the VRJP and the ERRW (on any graph).

\section{A simplification}

\subsection{Schur's lemma}
We will use Schur's decomposition several times in the paper. It is useful because it behaves nicely with the marginal and conditional laws of $\nu$.

\begin{lem}[Schur decomposition]\label{lem:Schur}
Let $H$ be a symmetric, positive definite matrix. Let $A,B,C$ be 3 matrices such that $H$ can be decomposed in bloc as such:
\[
H=
\left( \begin{matrix} A & B \\ \T{B} & C \end{matrix} \right).
\]
Its inverse is given by:
\[
H^{-1}= \left( \begin{matrix} A^{-1}+A^{-1}B(C-\T{B}A^{-1}B)^{-1} \T{B} A^{-1} & -A^{-1}B(C-\T{B}A^{-1}B)^{-1}\\ 
-(C-\T{B}A^{-1}B)^{-1}\T{B}A^{-1} & (C-\T{B}A^{-1}B)^{-1} \end{matrix} \right).
\]
\end{lem}

\subsection{Reduction to 2 points}

We want to show that we can reduce the problem to the study of $\nu_1$ and $\nu_2$, but first we need to prove a small lemma that will be useful in the following.
\begin{lem}\label{lem:connection}
Let $n$ be an integer, let $H\in M_n(\R)$ be a symmetric, positive definite matrix with non-positive off-diagonal coefficients. For any integers $1\leq i,j\leq n$, $H^{-1}(i,j)>0$ iff $i$ and $j$ are $H$-connected.
\end{lem}
\begin{proof}
Since $H$ is a symmetric, positive definite matrix, all its eigenvalues are positive reals. Let $\lambda^-$ be the smallest eigenvalue of $H$ and $\lambda^+$ its largest. Since $H$ is symmetric, all its diagonal coefficients $H(i,i)$ satisfy the inequality $\lambda^-\leq H(i,i) \leq \lambda^+$. This means that all the coefficients of $I_n-\frac{1}{\lambda^+}H$ are non-negative and its eigenvalues are between $0$ and $1-\frac{\lambda^-}{\lambda^+}<1$. This means that we have the following equality:
\[
H^{-1}=\frac{1}{\lambda^+}\left(I_n-\left(I_n-\frac{1}{\lambda^+}H\right)\right)^{-1}
=\frac{1}{\lambda^+}\sum\limits_{k\geq 0}\left(I_n-\frac{1}{\lambda^+}H\right)^k.
\]
For any integers $i,j$, $i$ and $j$ are $H$-connected iff there exists $m\geq 0$ such that $\left(I_n-\frac{1}{\lambda^+}H\right)^{m}>0$ (since all the coefficients of $I_n-\frac{1}{\lambda^+}H$ are non-negative). This means that $H^{-1}(i,j)>0$ iff $i$ and $j$ are $H$-connected.  
\end{proof}

We will use the following lemma to reduce our problem to the study of $\nu_1$ and $\nu_2$.

\begin{lem}\label{lem:downto2}
Let $n\in\N^*$ be an integer. Let $H^{11}\in M_{n}(\R)$ be a symmetric, positive definite matrix with non-positive off-diagonal coefficients. Let $H^{12}\in M_{n,2}(\R)$ be a matrix with only non-positive coefficients. We also define the matrix $\overline{H}^{12}\in M_{n,1}(\R)$ by:
\[
\overline{H}^{12}=  H^{12} \left(\begin{matrix} 1 \\ 1 \end{matrix}\right).
\]
Now let $H\in M_{n+2}(\R)$ and $\overline{H}\in M_{n+1}(\R)$ be two symmetric, positive definite matrices with non-positive off-diagonal coefficients such that they have the following bloc decomposition:
\[
H=\left(\begin{matrix} H^{11} & H^{12} \\ \T{H^{12}} & H^{22} \end{matrix}\right)
\text{ and } 
\overline{H}=\left(\begin{matrix} H^{11} & \overline{H}^{12} \\ \T{\overline{H}^{12}} & \overline{H}^{22} \end{matrix}\right).
\]
Let $G$ and $\overline{G}$ be the inverse of $H$ and $\overline{H}$ respectively. We use the same bloc decomposition:
\[
G=:\left(\begin{matrix} G^{11} & G^{12} \\ \T{G^{12}} & G^{22} \end{matrix}\right)
\text{ and } 
\overline{G}:=\left(\begin{matrix} G^{11} & \overline{G}^{12} \\ \T{\overline{G}^{12}} & \overline{G}^{22} \end{matrix}\right).
\]
For any vector $X\in\R^{n+2}$ we define the vector $\overline{X}\in\R^{n+1}$ by:
\[
\begin{aligned}
\forall i\in[\![1,n]\!],\ \overline{X}_i:=X_i \text{ and }\\
\overline{X}_{n+1}:=X_{n+1}+X_{n+2}.
\end{aligned}
\]
For any vectors $X^1,X^2\in[0,\infty)^{n+2}$ we can define:
\begin{itemize}
\item $\alpha_1(X^1)\geq 0$ and $\alpha_2(X^1)\geq 0$ that only depend on $X^1,H^{11}$ and $H^{12}$, 
\item $\alpha_1(X^2)\geq 0$ and $\alpha_2(X^2)\geq 0$ that only depend on $X^2,H^{11}$ and $H^{12}$,
\item $C(X^1,X^2)\geq 0$ that only depends on $X^1,X^2,H^{11}$ and $H^{12}$ (but not $H^{22}$),
\end{itemize}
such that:
\[
\begin{aligned}
\T{X^1} G X^2 =&  C(X^1,X^2) + \left(\begin{matrix} \alpha_1(X^1) & \alpha_2(X^1) \end{matrix}\right) 
G^{22}
\left(\begin{matrix} \alpha_1(X^2) \\ \alpha_2(X^2) \end{matrix}\right)\\[4pt]
\T{\overline{X}^1} \overline{G} \overline{X}^2 =& 
C(X^1,X^2) + (\alpha_1(X^1) +\alpha_2(X^1)) \overline{G}^{22} (\alpha_1(X^2) +\alpha_2(X^2)).
\end{aligned}
\]
\end{lem}
The previous lemma allows us to transform the expression $\T{X}^1 G X^{2}$ in the form $A + \T{Y}^1 G^{22} Y^2$. The properties of the family of law $\nu$ (\ref{lem:VRJPNu0}) tell us that the study of $G^{22}$ knowing $A,Y^1$ and $Y^2$ is the same as the study of $\nu_2$ for some parameters. This means that if we get some monotonicity for $\nu_2$ we should be able to get it back for $\nu_n$ for any $n$.
\begin{proof}[proof of lemma \ref{lem:downto2}]
First we look at $H$. Let $G$ be the inverse of $H$. We use the same bloc decomposition as for $H$:
\[
G=\left(\begin{matrix} G^{11} & G^{12} \\ G^{21} & G^{22} \end{matrix}\right),
\]
where $G^{11}\in M_n(\R),\ G^{12}\in M_{n,2}(\R),\ G^{21}\in M_{2,n}(\R)$ and $G^{22}\in M_{2}(\R)$.
By Schur decomposition \ref{lem:Schur} we have:
\[
\begin{aligned}
G=& \left( \begin{matrix} (H^{11})^{-1}+(H^{11})^{-1}H^{12} G^{22} \T{H^{12}} (H^{11})^{-1} & -(H^{11})^{-1}H^{12} G^{22}\\ 
-G^{22} \T{H^{12}}(H^{11})^{-1} & G^{22} \end{matrix} \right)\\
=&\left(\begin{matrix} I_n &  -(H^{11})^{-1}H^{12} \\ 0 & I_2 \end{matrix}\right)
\left(\begin{matrix} (H^{11})^{-1} &  0 \\ 0 & G^{22} \end{matrix}\right)
\left(\begin{matrix} I_n &  0 \\ -H^{21}(H^{11})^{-1} & I_2 \end{matrix}\right)
\end{aligned}
\]
By definition of $H$, all the coefficients of $-H^{12}$ are non-negative and all the coefficients of $(H^{11})^{-1}$ are also non-negative since $H^{11}$ is an M-matrix. This means that all the coefficients of $-(H^{11})^{-1}H^{12}$ are non-negative. Let $X^1,X^2\in \R^{n+2}$ be two vectors with the following bloc decomposition:
\[
X^1:=\left(\begin{matrix}X^{11} \\ X^{12} \end{matrix}\right) \text{ and }
X^2:=\left(\begin{matrix}X^{21} \\ X^{22} \end{matrix}\right),
\]
where $ X^{11},X^{21}\in \R^n$ and $ X^{12},X^{22}\in \R^2$. Let $M:=-H^{21}(H^{11})^{-1}$. We have:
\[
\begin{aligned}
\T{X^1} G X^2
=& \left(\begin{matrix}\T{X}^{11} & \T{X}^{12} \end{matrix}\right) 
\left(\begin{matrix} I_n &  -(H^{11})^{-1}H^{12} \\ 0 & I_2 \end{matrix}\right)
\left(\begin{matrix} (H^{11})^{-1} &  0 \\ 0 & G^{22} \end{matrix}\right)
\left(\begin{matrix} I_n &  0 \\ -H^{21}(H^{11})^{-1} & I_2 \end{matrix}\right)
\left(\begin{matrix}X^{21} \\ X^{22} \end{matrix}\right) \\
=& \left(\begin{matrix}\T{X}^{11} & \T{X}^{12} \end{matrix}\right) 
\left(\begin{matrix} I_n &  M \\ 0 & I_2 \end{matrix}\right)
\left(\begin{matrix} (H^{11})^{-1} &  0 \\ 0 & G^{22} \end{matrix}\right)
\left(\begin{matrix} I_n &  0 \\ M & I_2 \end{matrix}\right)
\left(\begin{matrix}X^{21} \\ X^{22} \end{matrix}\right) \\
=& \left(\begin{matrix} \T{X}^{11} & \T{X}^{11} \T{M} +\T{X}^{12} \end{matrix}\right) 
\left(\begin{matrix} (H^{11})^{-1} &  0 \\ 0 & G^{22} \end{matrix}\right)
\left(\begin{matrix}X^{21} \\ MX^{21} +X^{22} \end{matrix}\right)\\
=& \T{X}^{11}(H^{11})^{-1}X^{21}+ (\T{X}^{11} \T{M} +\T{X}^{12}) G^{22} (MX^{21} +X^{22})\\
=& \T{X}^{11}(H^{11})^{-1}X^{21}+ \T{(M X^{11}  +X^{12})} G^{22} (MX^{21} +X^{22}).
\end{aligned}
\]
Now we can define $\alpha_1(X^1),\alpha_2(X^1),\alpha_1(X^2)$ and $\alpha_2(X^2)$ by:
\[
\left(\begin{matrix}\alpha_1(X^1) \\ \alpha_2(X^1) \end{matrix}\right):= M X^{11}  +X^{12} \text{ and }
\left(\begin{matrix}\alpha_1(X^2) \\ \alpha_2(X^2) \end{matrix}\right):=MX^{21} +X^{22}.
\]
We also define $C(X^1,X^2)$ by $C(X^1,X^2):=\T{X}^{11}(H^{11})^{-1}X^{21}$. We get:
\[
\T{X^1} G X^2 =  C(X^1,X^2) + \left(\begin{matrix} \alpha_1(X^1) & \alpha_2(X^1) \end{matrix}\right) 
G^{22}
\left(\begin{matrix} \alpha_1(X^2) \\ \alpha_2(X^2) \end{matrix}\right).
\]
Similarly, we get:
\[
\begin{aligned}
\T{\overline{X}^1} \overline{G} \overline{X}^2
=&\T{X}^{11}(H^{11})^{-1}X^{21}+ \T{(-\overline{H}^{21}(H^{11})^{-1} X^{11}  +\overline{X}^{12})} \overline{G}^{22} (-\overline{H}^{21}(H^{11})^{-1}X^{21} +\overline{X}^{22})\\
=&C(X^1,X^2) + \left( \alpha_1(X^1) + \alpha_2(X^1) \right) 
\overline{G}^{22}
\left( \alpha_1(X^2) + \alpha_2(X^2) \right)
\end{aligned}
\]
\end{proof}

\section{The coupling}

\subsection{A change of variables}

When we look at $\nu_2$, instead of looking at the beta-field $(\beta_1,\beta_2)$ we will look at two other variables that will make our coupling and various calculations more explicit. In the following lemma we state this change of variables and some relevant properties of the new variables.
\begin{lem}\label{lem:coupling}
We set a parameter $\lambda\in [0,1]$ and a parameter $w\geq 0$ such that if $w=0$ then $\lambda\not\in\{0,1\}$. Let $W:=\left(\begin{matrix} 0 & w \\ w & 0\end{matrix}\right)$. Let $(\beta_1,\beta_2)$ be distributed according to $\nu_2^{W,0}$. We define the variables $\gamma$ and $Z$ by:
\[
\begin{aligned}
\gamma:=& \frac{1}{\left(\begin{matrix}\lambda & 1-\lambda\end{matrix}\right)\left(\begin{matrix}2\beta_1 & -w \\ -w & 2\beta_2\end{matrix}\right)^{-1}\left(\begin{matrix}\lambda \\ 1-\lambda\end{matrix}\right)} = \frac{4\beta_1\beta_2 -w^2}{2w\lambda(1-\lambda)+2\beta_2 \lambda^2 + 2 \beta_1 (1-\lambda)^2},\\
Z:=&\frac{2\beta_1 -\lambda^2 \gamma}{w+\lambda(1-\lambda)\gamma}.
\end{aligned}
\]
We have that both $Z$ and $\gamma$ are positive and:
\[
\begin{aligned}
2\beta_1 =& \lambda^2 \gamma + (w+\lambda(1-\lambda)\gamma)Z, \\
2\beta_2 =&   (1-\lambda)^2 \gamma + (w+\lambda(1-\lambda)\gamma)\frac{1}{Z}.
\end{aligned}
\]
The random variable $\gamma$ is the only random variable such that:
\[
\left(\begin{matrix}2\beta_1 & -w \\ -w & 2\beta_2 \end{matrix}\right)-\gamma\left(\begin{matrix} \lambda^2 & \lambda(1-\lambda) \\ \lambda(1-\lambda) & (1-\lambda)^2\end{matrix}\right)
\]
is of rank one. The law of $\gamma$ is that of a Gamma of parameter $(\frac{1}{2},\frac{1}{2})$. The law of $Z$, knowing $\gamma$ is given by:
\[
\frac{\sqrt{W+\lambda(1-\lambda)\gamma}}{\sqrt{2\pi}}\exp\left(- (W+\lambda(1-\lambda)\gamma)\frac{(z-1)^2}{2z}\right)\frac{1}{z}\left((1-\lambda)\sqrt{z}+\frac{\lambda}{\sqrt{z}}\right)1_{z>0}\dd z.
\]  
This law is a mixture of an inverse gaussian law and its inverse.\\
If $U$ is defined by $U:=\sqrt{Z}-\frac{1}{\sqrt{Z}}$, its density, knowing $\gamma$, is given by:
\[
\frac{\sqrt{w+\lambda(1-\lambda)\gamma}}{\sqrt{2\pi}}\exp\left(- (w+\lambda(1-\lambda)\gamma)\frac{u^2}{2}\right)\left(1-(2\lambda-1)\frac{u}{\sqrt{u^2+4}}\right)\dd u.
\]
This law is similar to a gaussian, in particular the law of $|U|$ is that of the absolute value of a gaussian.\\
We also have the following equality:
\[
\text{det}\left(\begin{matrix} 2\beta_1 & -w \\ -w & 2\beta_2 \end{matrix}\right)
=4\beta_1\beta_2-w^2=(w+\lambda(1-\lambda)\gamma)\gamma \left((1-\lambda)\sqrt{Z}+\frac{\lambda}{\sqrt{Z}}\right)^2.
\]
\end{lem}
The random variable $\gamma$ is a generalization of the random variable $\gamma$ defined in $\cite{VRJPNu0}$, in which it is only defined for $\lambda\in\{0,1\}$. It is used to make a link between the $\beta$-field and the VRJP starting at a specific point.
\begin{proof}
Let $\mathcal{H}\subset(0,\infty)^2$ be the set defined by:
\[
\mathcal{H}:=\left\{(b_1,b_2)\in(0,\infty)^2, \left(\begin{matrix}2b_1 & -w \\ -w & 2b_2\end{matrix}\right)>0\right\}.
\]
Let $f:(0,\infty)^2\mapsto \R^2$ be the function defined by:
\[
f(c,z):= \left(\frac{\lambda^2 c + (w+\lambda(1-\lambda)c)z}{2},    \frac{(1-\lambda)^2 c + (w+\lambda(1-\lambda)c)\frac{1}{z}}{2}\right).
\]
First we need to check that $f\left((0,\infty)^2\right)\subset \mathcal{H}$. First, $\frac{\lambda^2 c + (w+\lambda(1-\lambda)c)z}{2}>0$ and $\frac{(1-\lambda)^2 c + (w+\lambda(1-\lambda)c)\frac{1}{z}}{2}>0$. Then:
\[
4\frac{\lambda^2 c + (w+\lambda(1-\lambda)c)z}{2}\frac{(1-\lambda)^2 c + (w+\lambda(1-\lambda)c)\frac{1}{z}}{2} -w^2 >wz w\frac{1}{z} -w^2>0.
\]
This means that $f\left((0,\infty)^2\right)\subset \mathcal{H}$. \\
Now we need a small result on matrices that will make calculations on $f$ simpler. Let $Y:=\left(\begin{matrix} \lambda \\ 1-\lambda \end{matrix}\right)$. For any $(a_1,a_2)\in\mathcal{H}$ and $s\in\R$, we have:
\[
\begin{aligned}
\text{det}\left(\left(\begin{matrix}2a_1 & -w \\ -w & 2a_2\end{matrix}\right)-sY\T{Y}\right)
=&\text{det}\left(\begin{matrix}2a_1 & -w \\ -w & 2a_2\end{matrix}\right)\text{det}\left(I_2-s\left(\begin{matrix}2a_1 & -w \\ -w & 2a_2\end{matrix}\right)^{-1}Y\T{Y}\right)\\
=&\text{det}\left(\begin{matrix}2a_1 & -w \\ -w & 2a_2\end{matrix}\right)\text{det}\left(1-s\T{Y}\left(\begin{matrix}2a_1 & -w \\ -w & 2a_2\end{matrix}\right)^{-1}Y\right)\\
=&\text{det}\left(\begin{matrix}2a_1 & -w \\ -w & 2a_2\end{matrix}\right)\left(1-s\T{Y}\left(\begin{matrix}2a_1 & -w \\ -w & 2a_2\end{matrix}\right)^{-1}Y\right).
\end{aligned}
\] 
This means that 
\[
\text{det}\left(\left(\begin{matrix}2a_1 & -w \\ -w & 2a_2\end{matrix}\right)-sY\T{Y}\right)=0\Leftrightarrow s= \frac{1}{\T{Y}\left(\begin{matrix}2a_1 & -w \\ -w & 2a_2\end{matrix}\right)^{-1}Y}.
\]
Now we notice that if $(b_1,b_2):=f(c,z)$ then
\[
\left(\begin{matrix}2b_1 & -w \\ -w & 2b_2\end{matrix}\right)-cY\T{Y}=\left(\begin{matrix}(w+\lambda(1-\lambda)c)z & -(w+\lambda(1-\lambda)c) \\ -(w+\lambda(1-\lambda)c) & (w+\lambda(1-\lambda)c)\frac{1}{z}\end{matrix}\right),
\]
which is of rank one, and the eigenvector for the non-zero eigenvalue is $\left(\begin{matrix} \sqrt{z} \\ \frac{1}{\sqrt{z}}\end{matrix}\right)$.\\
Therefore if we know that $(b_1,b_2)=f(c,z)$ then
\[
\begin{aligned}
c&= \frac{1}{\T{Y}\left(\begin{matrix}2b_1 & -w \\ -w & 2b_2\end{matrix}\right)^{-1}Y}
=\frac{4b_1b_2-w^2}{2b_2\lambda^2+2b_1(1-\lambda)^2+2\lambda(1-\lambda)w},\text{ and}\\
z&=\frac{2b_1-\lambda^2c}{w+\lambda(1-\lambda)\gamma}=\frac{w+\lambda(1-\lambda)\gamma}{2b_2-(1-\lambda)^2c}.
\end{aligned}
\]
This means that $f$ is injective and its inverse is the one we want. Conversely, $f$ is surjective by using the same formula.\\
The Jacobian $J_f$ of the change of variables $f$ is equal to:
\[
J_f(c,z)=\left(
\begin{matrix}
\left(\lambda^2 + \lambda(1-\lambda)z\right)\frac{1}{2} & \left((1-\lambda)^2 + \lambda(1-\lambda)\frac{1}{z}\right)\frac{1}{2}\\[+4pt]
(w+\lambda(1-\lambda)c)\frac{1}{2} & -(w+\lambda(1-\lambda)c)\frac{1}{2z^2}
\end{matrix}
\right)
\] 
and therefore the determinant $D_f$ of the Jacobian is equal to :
\[
\begin{aligned}
D_f(c,z)=&\frac{w+\lambda(1-\lambda)c}{4}\left((1-\lambda)^2 + \lambda(1-\lambda)\frac{1}{z}+\lambda^2\frac{1}{z^2} + \lambda(1-\lambda)\frac{1}{z}\right)\\
=&\frac{w+\lambda(1-\lambda)c}{4}\left(1-\lambda+\frac{\lambda}{z}\right)^2\\
=&\frac{w+\lambda(1-\lambda)c}{4}\frac{1}{z}\left((1-\lambda)\sqrt{z}+\frac{\lambda}{\sqrt{z}}\right)^2.
\end{aligned}
\]
Now we can change variables $(\beta_1,\beta_2)$ such that $H_{\beta}:=\left(\begin{matrix}2b_1 & -w \\ -w & 2b_2 \end{matrix}\right) >0$ into variables $(\gamma,z)$ defined by:
\[
\begin{aligned}
\gamma:=& \frac{4\beta_1\beta_2 -w^2}{2w\lambda(1-\lambda)+2\beta_2 \lambda^2 + 2 \beta_1 (1-\lambda)^2},\\
z:=&\frac{2\beta_1 -\lambda^2 \gamma}{w+\lambda(1-\lambda)\gamma}.
\end{aligned}
\]
We need to make a few calculations before we can express the law of $(\gamma,Z)$. First we have, for any $(c,z)\in (0,\infty)^2$, with $(b_1,b_2):=f(c,z)$:
\[
\begin{aligned}
&4b_1b_2 -w^2 \\
=& \left( (w+\lambda(1-\lambda)c)z+\lambda^2 c\right)\left( (w+\lambda(1-\lambda)c)\frac{1}{z}+(1-\lambda )^2 c\right) - w^2\\
=&  (w+\lambda(1-\lambda)c)^2 + (w+\lambda(1-\lambda)c)\left(\lambda^2 c\frac{1}{z}+(1-\lambda )^2c z\right) +\lambda^2(1-\lambda)^2c^2-w^2 \\
=&  (\lambda(1-\lambda)c)^2 + 2w\lambda(1-\lambda)c + (w+\lambda(1-\lambda)c)\left(\lambda^2 c\frac{1}{x}+(1-\lambda )^2c x\right) +\lambda^2(1-\lambda)^2c^2 \\
=&2c(w+\lambda(1-\lambda)c) + (w+\lambda(1-\lambda)c)c\left(\lambda^2 \frac{1}{z}+(1-\lambda )^2 z\right) +\lambda^2(1-\lambda)^2c^2\\
=& (w+\lambda(1-\lambda)c)c\left(\lambda^2 \frac{1}{z}+(1-\lambda )^2 z+2\right) \\
=&(w+\lambda(1-\lambda)c)c\left((1-\lambda)\sqrt{z}+\frac{\lambda}{\sqrt{z}}\right)^2 .
\end{aligned}
\]
Therefore we get:
\[
\frac{D_f(c,z)}{\sqrt{4b_1b_2-w^2}}
=\frac{\sqrt{w+\lambda(1-\lambda)c}}{4\sqrt{c}}\frac{1}{z}\left((1-\lambda)\sqrt{z}+\frac{\lambda}{\sqrt{z}}\right).
\]
We also have the following equality:
\[
\begin{aligned}
b_1+b_2-w
=& \lambda^2 \frac{c}{2} + (w+\lambda(1-\lambda)c)\frac{z}{2} + (1-\lambda )^2 \frac{c}{2} + (w+\lambda(1-\lambda)c)\frac{1}{2z}\\
 &\ \ \ \ - ((w+\lambda(1-\lambda)c)-\lambda(1-\lambda)c) \\
=& (\lambda^2+(1-\lambda)^2+2)\frac{c}{2} + \frac{1}{2}(w+\lambda(1-\lambda)c)\left(z+\frac{1}{z}-2\right)\\
=& \frac{c}{2} + \frac{1}{2}(w+\lambda(1-\lambda)c)\frac{1}{z}\left(z-1\right)^2.
\end{aligned}
\]
And therefore we get the following joint law for $\gamma$ and $Z$ ($c$ represents $\gamma$ and $z$ represents $Z$):
\[
\frac{2}{\pi}\frac{\sqrt{w+\lambda(1-\lambda)c}}{4\sqrt{c}}\frac{1}{z}\left((1-\lambda)\sqrt{z}+\frac{\lambda}{\sqrt{z}}\right)
\exp\left(-\frac{c}{2} - (w+\lambda(1-\lambda)c)\frac{(z-1)^2}{2z}\right)\dd z \dd c.
\]
In particular, the law of $Z$, knowing $\gamma$, is given by
\[
\frac{\sqrt{w+\lambda(1-\lambda)\gamma}}{\sqrt{2\pi}}\exp\left(- (w+\lambda(1-\lambda)\gamma)\frac{(z-1)^2}{2z}\right)\frac{1}{z}\left((1-\lambda)\sqrt{z}+\frac{\lambda}{\sqrt{z}}\right)\dd z.
\]
It is indeed a density since it is a mixture of an inverse gaussian and the inverse of an inverse gaussian. Now, we can look at the law of $U$. By definition, $U=\sqrt{Z}-\frac{1}{\sqrt{Z}}$. This means that $\sqrt{Z}=\frac{\sqrt{U^2+4}+U}{2}$ and $\frac{1}{\sqrt{Z}}=\frac{\sqrt{U^2+4}-U}{2}$. We therefore have $Z=\frac{U^2+2+U\sqrt{U^2+4}}{2}$. The density of $U$ is thus:
\[
\begin{aligned}
& \frac{1}{2}\left(2u+\sqrt{u^2+4} + \frac{u^2}{\sqrt{u^2+4}}\right)\frac{\sqrt{w+\lambda(1-\lambda)\gamma}}{\sqrt{2\pi}}\exp\left(- (w+\lambda(1-\lambda)\gamma)\frac{u^2}{2}\right)\\
&\ \ \ \ \times\frac{2}{u^2+2+u\sqrt{u^2+4}}\left((1-\lambda)\frac{\sqrt{u^2+4}+u}{2}+\lambda\frac{\sqrt{u^2+4}-u}{2}\right)\dd u\\
=& \frac{2u\sqrt{u^2+4}+2u^2+4}{2\sqrt{u^2+4}}\frac{\sqrt{w+\lambda(1-\lambda)\gamma}}{\sqrt{2\pi}}\exp\left(- (w+\lambda(1-\lambda)\gamma)\frac{u^2}{2}\right)\\
&\ \ \ \ \times\frac{2}{u^2+2+u\sqrt{u^2+4}}\left((1-\lambda)\frac{\sqrt{u^2+4}+u}{2}+\lambda\frac{\sqrt{u^2+4}-u}{2}\right)\dd u\\
=& \frac{\sqrt{w+\lambda(1-\lambda)\gamma}}{\sqrt{2\pi}}\exp\left(- (w+\lambda(1-\lambda)\gamma)\frac{u^2}{2}\right)\left(1-(2\lambda-1)\frac{u}{\sqrt{u^2+4}}\right)\dd u.
\end{aligned}
\]
\end{proof}

\subsection{The tilted gaussian law}

\begin{defin}
For any $(K,\delta)\in(0,\infty)\times\left[-1,1\right]$ we define the tilted gaussian law $\tilde{\mathcal{N}}(K,\delta)$ by the following density:
\[
\sqrt{\frac{K}{2\pi}}\exp\left(-\frac{Ku^2}{2}\right)\left(1+\delta\frac{u}{\sqrt{u^2+4}}\right)\dd u.
\]
It is indeed a density since it is the density of a gaussian plus an antisymmetric term that is smaller than the gaussian term.
\end{defin}

\begin{lem}\label{lem:tiltedgaussianesp}
Set $K>0$ and $\delta,\delta^{\prime}\in[-1,1]$. Let $U$ be a random variable distributed according to $\tilde{\mathcal{N}}(K,\delta)$. We have the following equality:
\[
\E\left(\frac{1+\delta^{\prime}\frac{U}{\sqrt{(U)^2+4}}}{1+\delta\frac{U}{\sqrt{(U)^2+4}}}\right)=1.
\]
\end{lem}
\begin{proof}
We have:
\[
\begin{aligned}
\E\left(\frac{1+\delta^{\prime}\frac{U}{\sqrt{(U)^2+4}}}{1+\delta\frac{U}{\sqrt{(U)^2+4}}}\right)
=&\int\limits_{u\in\R}\sqrt{\frac{K}{2\pi}}\exp\left(-\frac{Ku^2}{2}\right)\left(1+\delta\frac{u}{\sqrt{u^2+4}}\right)\left(\frac{1+\delta^{\prime}\frac{u}{\sqrt{(u)^2+4}}}{1+\delta\frac{u}{\sqrt{(u)^2+4}}}\right)\dd u\\
=&\int\limits_{u\in\R}\sqrt{\frac{K}{2\pi}}\exp\left(-\frac{Ku^2}{2}\right)\left(1+\delta^{\prime}\frac{u}{\sqrt{(u)^2+4}}\right)\dd u\\
=&1.
\end{aligned}
\]
\end{proof}
\begin{lem}\label{lem:tiltedgaussian}
Let $0< K^-\leq K^+$. Set $\delta\in[-1,1]$. There exists two random variables $U^-$ and $U^+$ distributed according to $\tilde{\mathcal{N}}(K^-,\delta)$ and $\tilde{\mathcal{N}}(K^+,\delta)$ respectively such that:
\[
\forall \delta^{\prime}\in[-1,1],\ 
\E\left(\frac{1+\delta^{\prime}\frac{U^-}{\sqrt{(U^-)^2+4}}}{1+\delta\frac{U^-}{\sqrt{(U^-)^2+4}}}|U^+\right)
=\frac{1+\delta^{\prime}\frac{U^+}{\sqrt{(U^+)^2+4}}}{1+\delta\frac{U^+}{\sqrt{(U^+)^2+4}}},
\] 
and
\[
K^{-}(U^-)^2=K^{+}(U^+)^2 \text{ a.s.}
\]
\end{lem}
\begin{proof}
Let $K:=\sqrt{\frac{K^+}{K^-}}$. Let $U^+$ be a random variable distributed according to $\tilde{\mathcal{N}}(K^+,\delta)$. First we define the random variables $V^+$ and $V^-$ by:
\[
\begin{aligned}
V^+:=&\frac{U^+}{\sqrt{(U^+)^2+4}}\\
V^-:=&\frac{KU^+}{\sqrt{K^2(U^+)^2+4}}.
\end{aligned}
\]
We notice that $0\leq |V^+|\leq |V^-|<1$. Let $p_1,p_2\in\R$ be defined by:
\[
p^+:=\frac{1}{2}\left(1+\frac{V^+}{V^-}\right)\frac{1+\delta V^-}{1+\delta V^+} \text{ and }
p^-:=\frac{1}{2}\left(1-\frac{V^+}{V^-}\right)\frac{1-\delta V^-}{1+\delta V^+}.
\]
Both $p^+$ and $p^-$ are non-negative. We also have:
\[
\begin{aligned}
p^+ + p^-=&
\frac{1}{2}\left(1+\frac{V^+}{V^-}\right)\frac{1+\delta V^-}{1+\delta V^+}
+\frac{1}{2}\left(1-\frac{V^+}{V^-}\right)\frac{1-\delta V^-}{1+\delta V^+}\\
=&\frac{1+\delta V^- + 1 - \delta V^- +\frac{V^+}{V^-}\left(1+\delta V^- - 1 + \delta V^-\right)}{2(1+\delta V^+)}\\
=&\frac{2 +\frac{V^+}{V^-}2\delta V^-}{2(1+\delta V^+)}=1.
\end{aligned}
\]
 Now, let $U^{-}$ the random variable be such that knowing $U^+$:
\[
U^-:=\left\{ \begin{matrix}
KU^+ \text{ with probability } p^+ \\
-KU^+ \text{ with probability } p^-
\end{matrix}\right. .
\]
Now we want to show that $U^{-}$ is distributed according to $\tilde{\mathcal{N}}(K^-,\delta)$. We have, for any test function $f$:
\[
\begin{aligned}
\E\left(f(U^-)\right)
=&\E\left(\E\left(f(U^-)|U^+\right)\right)\\
=&\E\left(\frac{1}{2}\left(1+\frac{V^+}{V^-}\right)\frac{1+\delta V^-}{1+\delta V^+}f\left(KU^+\right)
+\frac{1}{2}\left(1-\frac{V^+}{V^-}\right)\frac{1-\delta V^-}{1+\delta V^+}f\left(-KU^+\right)\right).
\end{aligned}
\]
First we get:
\[
\begin{aligned}
&\E\left(\frac{1}{2}\left(1+\frac{V^+}{V^-}\right)\frac{1+\delta V^-}{1+\delta V^+}f\left(KU^+\right)\right)\\
=&\int\limits_{u\in\R}\sqrt{\frac{K^+}{2\pi}}\exp\left(-\frac{K^+u^2}{2}\right)\left(1+\delta\frac{u}{\sqrt{u^2+4}}\right)\left(\frac{1}{2}\left(1+\frac{\sqrt{K^2u^2+4}}{K\sqrt{u^2+4}}\right)\frac{1+\delta K\frac{u}{\sqrt{K^2u^2+4}}}{1+\delta\frac{u}{\sqrt{u^2+4}}}f\left(Ku\right)\right)\dd u\\
=&\int\limits_{u\in\R}\sqrt{\frac{K^+}{2\pi}}\exp\left(-\frac{K^+u^2}{2}\right)\left(1+\delta K\frac{u}{\sqrt{K^2u^2+4}}\right)\left(\frac{1}{2}\left(1+\frac{\sqrt{K^2u^2+4}}{K\sqrt{u^2+4}}\right)f\left(Ku\right)\right)\dd u\\
=&\int\limits_{u\in\R}\sqrt{\frac{K^-}{2\pi}}\exp\left(-\frac{K^-u^2}{2}\right)\left(1+\delta \frac{u}{\sqrt{u^2+4}}\right)\left(\frac{1}{2}\left(1+\frac{\sqrt{u^2+4}}{\sqrt{u^2+4K}}\right)f\left(u\right)\right)\dd u.
\end{aligned}
\]
Similarly, we have:
\[
\begin{aligned}
&\E\left(\frac{1}{2}\left(1-\frac{V^+}{V^-}\right)\frac{1-\delta V^-}{1+\delta V^+}f\left(-KU^+\right)\right)\\
=&\int\limits_{u\in\R}\sqrt{\frac{K^+}{2\pi}}\exp\left(-\frac{K^+u^2}{2}\right)\left(1+\delta\frac{u}{\sqrt{u^2+4}}\right)\left(\frac{1}{2}\left(1-\frac{\sqrt{K^2u^2+4}}{K\sqrt{u^2+4}}\right)\frac{1-\delta K\frac{u}{\sqrt{K^2u^2+4}}}{1+\delta\frac{u}{\sqrt{u^2+4}}}f\left(-Ku\right)\right)\dd u\\
=&\int\limits_{u\in\R}\sqrt{\frac{K^+}{2\pi}}\exp\left(-\frac{K^+u^2}{2}\right)\left(1-\delta K\frac{u}{\sqrt{K^2u^2+4}}\right)\left(\frac{1}{2}\left(1-\frac{\sqrt{K^2u^2+4}}{K\sqrt{u^2+4}}\right)f\left(-Ku\right)\right)\dd u\\
=&\int\limits_{u\in\R}\sqrt{\frac{K^-}{2\pi}}\exp\left(-\frac{K^-u^2}{2}\right)\left(1+\delta \frac{u}{\sqrt{u^2+4}}\right)\left(\frac{1}{2}\left(1-\frac{\sqrt{u^2+4}}{\sqrt{u^2+4K}}\right)f\left(u\right)\right)\dd u.
\end{aligned}
\]
If we put both equalities together, we get for any test function $f$:
\[
\E\left(f(U^-)\right)
=\int\limits_{u\in\R}\sqrt{\frac{K^-}{2\pi}}\exp\left(-\frac{K^-u^2}{2}\right)\left(1+\delta \frac{u}{\sqrt{u^2+4}}\right)f(u)\dd u.
\]
This means that $U^-$ is indeed distributed according to $\tilde{\mathcal{N}}(K^-,\delta)$. Now we only need to show that $U^+$ and $U^-$ satisfy the equality we want. First we notice that for any $x\in (-1,1)$:
\[
\frac{1+\delta^{\prime} x}{1+\delta x}= 1 + (\delta^{\prime}-\delta)\frac{x}{1+\delta x}.
\]
This means that we only need to show that:
\[
\E\left(\frac{\frac{U^-}{\sqrt{(U^-)^2+4}}}{1+\delta \frac{U^-}{\sqrt{(U^-)^2+4}}}|U^+\right)
=\frac{\frac{U^+}{\sqrt{(U^+)^2+4}}}{1+\delta \frac{U^+}{\sqrt{(U^+)^2+4}}}.
\]
Which is the same as showing:
\[
\E\left(\frac{\frac{U^-}{\sqrt{(U^-)^2+4}}}{1+\delta \frac{U^-}{\sqrt{(U^-)^2+4}}}|U^+\right)
=\frac{V^+}{1+\delta V^+}.
\]
By definition of $U^-,V^-$ and $V^+$, we have:
\[
\begin{aligned}
&\E\left(\frac{\frac{U^-}{\sqrt{(U^-)^2+4}}}{1+\delta \frac{U^-}{\sqrt{(U^-)^2+4}}}|U^+\right)\\
=&\frac{1}{2}\left(1+\frac{V^+}{V^-}\right)\frac{1+\delta V^-}{1+ \delta V^+}\frac{\frac{KU^+}{\sqrt{(KU^+)^2+4}}}{1+\delta \frac{KU^+}{\sqrt{(KU^+)^2+4}}}
+\frac{1}{2}\left(1-\frac{V^+}{V^-}\right)\frac{1-\delta V^-}{1+ \delta V^+}\frac{\frac{-KU^+}{\sqrt{(-KU^+)^2+4}}}{1+\delta \frac{-KU^+}{\sqrt{(-KU^+)^2+4}}}\\
=&\frac{1}{2}\left(1+\frac{V^+}{V^-}\right)\frac{1+\delta V^-}{1+ \delta V^+}\frac{V^-}{1+\delta V^-}
+\frac{1}{2}\left(1-\frac{V^+}{V^-}\right)\frac{1-\delta V^-}{1+ \delta V^+}\frac{-V^-}{1-\delta V^-}\\
=&\frac{1}{2}\left(1+\frac{V^+}{V^-}\right)\frac{V^-}{1+ \delta V^+}
+\frac{1}{2}\left(1-\frac{V^+}{V^-}\right)\frac{-V^-}{1+ \delta V^+}\\
=&\frac{V^+}{1+\delta V^+}
\end{aligned}
\]
\end{proof}

\begin{lem}\label{lem:martingale}
Set $w>0$ and $W:=\left(\begin{matrix} 0 & w \\ w & 0\end{matrix}\right)$. Now set 2 parameters $\lambda,\theta\in[0,1]$. Let $(\beta_1,\beta_2)$ be distributed according to $\nu_2^{W,0}$. Let $H_{\beta}$ be the random matrix defined by:
\[
H_{\beta}:=\left(\begin{matrix} 2\beta_1 & -w \\ -w & 2\beta_2 \end{matrix}\right).
\] 
Let $G_{\beta}$ be the inverse of $H_{\beta}$. We define the random variables $\gamma$ and $Z$ by:
\[
\begin{aligned}
\gamma:=& \frac{4\beta_1\beta_2 -w^2}{2w\lambda(1-\lambda)+2\beta_2 \lambda^2 + 2 \beta_1 (1-\lambda)^2},\\
Z:=&\frac{2\beta_1 -\lambda \gamma}{w+\lambda(1-\lambda)\gamma}.
\end{aligned}
\]
We have:
\[
\left(\begin{matrix} \lambda & (1-\lambda)\end{matrix}\right) G_{\beta} \left(\begin{matrix} \theta \\ (1-\theta) \end{matrix}\right)
=\frac{\theta\frac{1}{\sqrt{Z}}+(1-\theta) \sqrt{Z}}{\gamma\left((1-\lambda)\sqrt{Z} + \lambda\frac{1}{\sqrt{Z}}\right)}.
\]
\end{lem}
\begin{proof}
First, by lemma \ref{lem:coupling} we have:
\[
\begin{aligned}
2\beta_1 =& \left(w+\lambda(1-\lambda) \gamma \right) Z + \lambda^2 \gamma,\\
2\beta_2 =& \left(w+\lambda(1-\lambda) \gamma \right) \frac{1}{Z} + (1-\lambda)^2 \gamma,\\
 w =& \left(w+\lambda(1-\lambda) \gamma \right) - \lambda(1-\lambda) \gamma.
\end{aligned}
\]
To simplify notations, let $\tilde{w}$ be the random variable defined by $\tilde{w}:=w+\lambda(1-\lambda) \gamma $. A quantity that will be important in the following is the determinant of $H_{\beta}$: $4\beta_1\beta_2- w^2$. By lemma \ref{lem:coupling}, we have:
\[
4\beta_1\beta_2- w^2
=\tilde{w}\gamma\left((1-\lambda)\sqrt{Z} + \lambda\frac{1}{\sqrt{Z}}\right)^2.
\]
We know that :
\[
G_{\beta}(1,1)=\frac{2\beta_2}{4\beta_1\beta- w^2},\ 
G_{\beta}(2,2)=\frac{2\beta_1}{4\beta_1\beta- w^2} \text{ and }
G_{\beta}(1,2)=G_{\beta}(2,1)=\frac{w}{4\beta_1\beta- w^2}.
\]
Therefore:
\[
\left(\begin{matrix} \lambda & 1-\lambda\end{matrix}\right) G_{\beta} \left(\begin{matrix} \theta \\ 1-\theta \end{matrix}\right)\\
= \frac{\lambda\theta2\beta_2+(\lambda(1-\theta)+(1-\lambda)\theta)w+(1-\lambda)(1-\theta)2\beta_1}{4\beta_1\beta_2-w^2}.
\]
Now we also have:
\[
\begin{aligned}
&\lambda\theta 2\beta_2+(\lambda(1-\theta)+(1-\lambda)\theta)w+(1-\lambda)(1-\theta)2\beta_1\\
=&\lambda\theta \left(\tilde{w}\frac{1}{Z}+(1-\lambda)^2\gamma\right)+(\lambda(1-\theta)+(1-\lambda)\theta)(\tilde{w}-\lambda(1-\lambda)\gamma)+(1-\lambda)(1-\theta)\left(\tilde{w}Z+\lambda^2\gamma\right)\\
=&\lambda\theta \tilde{w}\frac{1}{Z}+(\lambda(1-\theta)+(1-\lambda)\theta)\tilde{W}+(1-\lambda)(1-\theta)\tilde{w}Z\\
=&\tilde{w}\left(\lambda\frac{1}{\sqrt{Z}}+(1-\lambda) \sqrt{Z}\right)\left(\theta\frac{1}{\sqrt{Z}}+(1-\theta) \sqrt{Z}\right).
\end{aligned}
\]
We therefore get:
\[
\begin{aligned}
\left(\begin{matrix} \lambda & 1-\lambda \end{matrix}\right) G_{\beta} \left(\begin{matrix} \theta \\ 1-\theta \end{matrix}\right)
=& \frac{\tilde{w}\left(\lambda\frac{1}{\sqrt{Z}}+(1-\lambda) \sqrt{z}\right)\left(\theta\frac{1}{\sqrt{Z}}+(1-\theta) \sqrt{Z}\right)}{\tilde{w}\gamma\left((1-\lambda)\sqrt{Z} + \lambda\frac{1}{\sqrt{Z}}\right)^2}\\
=&\frac{\theta\frac{1}{\sqrt{Z}}+(1-\theta) \sqrt{Z}}{\gamma\left((1-\lambda)\sqrt{Z} + \lambda\frac{1}{\sqrt{Z}}\right)}.
\end{aligned}
\]
\end{proof}

\section{Main theorem}

Some of the results are based on some manipulations on graph, mostly we will quotient graphs. We remind the reader of the definition of the quotient of a graph by one of its subset. We also add the notion of weight for these quotients.
\begin{defin}
Let $\mathcal{G}=(V,E)$ be a locally finite, non-directed graph. Let $(W_e)_{e\in E}$ be a family of weights on the edges of $\mathcal{G}=(V,E)$. Let $A$ be a subset of $V$. The quotient $(\tilde{V}^A,\tilde{E}^A),\tilde{W}^A$ of the weighted graph $\mathcal{G},W$ by the subset of vertices $A$ is defined by:
\[
\begin{aligned}
&\tilde{V}^A:= V\backslash A \cup \{x_A\}  \\
&\tilde{E}^A:= \{ \{x,y\}\in E, x,y\in V\backslash A  \} \cup \{ \{x_A,y\} \in \left(\tilde{V}_A\right)^2, \exists x \in A, \{x,y\}\in E  \}\\
&\forall \{x,y\} \in \tilde{E}^A, x,y\not\in A, W^A_{\{x,y\}}:=W_{\{x,y\}},\\
&\forall x\in \tilde{V}^A \backslash \{x_A\} \text{ such that } \{x_A,x\}\in \tilde{E}^A, W^A_{\{x_A,a\}}:=\sum\limits_{y\in A} 1_{\{x,y\}\in E} W_{\{x,y\}}. 
\end{aligned}
\]
\end{defin}

We can now prove our main theorem.

\begin{proof}[proof of theorem \ref{maintheo}]
According to proposition \ref{lem:VRJPNu0}, the marginal law of $(\beta_i)_{1\leq i\leq n}$ is the same under $\nu_{n+2}^{W^-,0},\nu_{n+2}^{W^+,0}$ and $\nu_{n+1}^{W^{\infty},0}$ and is equal to $\nu_{n}^{W,\eta}$ for some $\eta\in\R^n$. Let $H$ be distributed according to $\tilde{\nu}_{n}^{W,\eta}$. Let $K\in[0,+\infty)$ be the random variable defined by 
\[
K:=\T{W}^2 H^{-1} W^1,
\] 
and $\tilde{K}$ the random matrix defined by:
\[
\tilde{K}=\left(\begin{matrix} 0 & K \\ K & 0\end{matrix}\right).
\]
Set a vector $X^1\in [0,\infty)^{n+2}$. Let $\alpha_1(X^1)$ and $\alpha_2(X^1)$ be the numbers defined in lemma \ref{lem:downto2} and $\alpha(X^1):=\alpha_1(X^1)+\alpha_2(X^1)$. Let $\lambda\in[0,1]$ be the random variable defined by
\[
\lambda:=\left\{\begin{matrix}\frac{\alpha_1(X^1)}{\alpha_1(X^1)+\alpha_2(X^1)} &\text{ if }\alpha_1(X^1)+\alpha_2(X^1)\not=0 \\ 0 &\text{ otherwise} \end{matrix}\right. ,
\]
and $\delta\in[-1,1]$ the random variable defined by $\delta:=2\lambda-1$.\\
If $n+1$ and $n+2$ are $H^-$-connected then $K+w^->0$. Let $\gamma$ be a random variable distributed according to a $\Gamma\left(\frac{1}{2}\right)$ distribution. Now let $U^-$ and $U^+$ be two random variables distributed according to $\tilde{\mathcal{N}}\left(K+w^-,\delta\right)$ and $\tilde{\mathcal{N}}\left(K+w^+,\delta\right)$ respectively and such that 
\[
\forall \delta^{\prime}\in[-1,1],\ 
\E\left(\frac{1+\delta^{\prime}\frac{U^-}{\sqrt{(U^-)^2+4}}}{1+\delta\frac{U^-}{\sqrt{(U^-)^2+4}}}|U^+\right)
=\frac{1+\delta^{\prime}\frac{U^+}{\sqrt{(U^+)^2+4}}}{1+\delta\frac{U^+}{\sqrt{(U^+)^2+4}}}.
\] 
Such two random variables exist by lemma \ref{lem:tiltedgaussian}. We define the positive random variables $Z^-$ and $Z^+$ by:
\[
U^-=\sqrt{Z^-}-\frac{1}{\sqrt{Z^-}} \text{ and }
U^+=\sqrt{Z^+}-\frac{1}{\sqrt{Z^+}}.
\]
Now, we define the random variables $\tilde{\beta}^{-}_{n+1},\tilde{\beta}^{-}_{n+2},\tilde{\beta}^{+}_{n+1}$ and $\tilde{\beta}^{+}_{n+2}$ by:
\[
\begin{aligned}
2\tilde{\beta}^{-}_{n+1}=&\left(K+w^-+\lambda(1-\lambda)\gamma\right)Z^- +\lambda^2\gamma \\
2\tilde{\beta}^{-}_{n+2}=&\left(K+w^-+\lambda(1-\lambda)\gamma\right)Z^- +(1-\lambda)^2\gamma \\
2\tilde{\beta}^{+}_{n+1}=&\left(K+w^++\lambda(1-\lambda)\gamma\right)Z^+ +\lambda^2\gamma \\
2\tilde{\beta}^{+}_{n+2}=&\left(K+w^++\lambda(1-\lambda)\gamma\right)Z^+ +(1-\lambda)^2\gamma.
\end{aligned}
\]
Let $\tilde{K}^-$ and $\tilde{K}^+$ be the matrices defined by:
\[
\tilde{K}^-:=\left(\begin{matrix}  0 & w^-+K \\ w^-+K & 0 \end{matrix}\right) \text{ and } \tilde{K}^+:=\left(\begin{matrix}  0 & w^++K \\ w^++K & 0 \end{matrix}\right).
\]
By lemma \ref{lem:coupling}, knowing $K$ and $\delta$, $(\tilde{\beta}^{-}_{n+1},\tilde{\beta}^{-}_{n+2})$ and $(\tilde{\beta}^{-}_{n+1},\tilde{\beta}^{-}_{n+2})$are distributed according to $\nu_2^{\tilde{K}^-,0}$ and $\nu_2^{\tilde{K}^+,0}$ respectively. Now we can define the matrices $H^-,H^+$ and $H^{\infty}$ by bloc:
\[
H^-=\left(\begin{matrix} H & -W^1 & -W^2 \\ -\T{W^1}& 2\tilde{\beta}^{-}_{n+1}+\T{W}^{1} H^{-1} W^{1}& -w^-\\ -\T{W^2}& -w^-& 2\tilde{\beta}^{-}_{n+2}+\T{W}^{2} H^{-1} W^{2}\end{matrix}\right),
\]
\[
H^+=\left(\begin{matrix} H & -W^1 & -W^2 \\ -\T{W^1}& 2\tilde{\beta}^{-}_{n+1}+\T{W}^{1} H^{-1} W^{1}& -w^+\\ -\T{W^2}& -w^+& 2\tilde{\beta}^{+}_{n+2}+\T{W}^{2} H^{-1} W^{2}\end{matrix}\right),
\]
\[
H^{\infty}=\left(\begin{matrix} H & -W^1 -W^2 \\ -\T{W^1}-\T{W^2}& \gamma + \left(\T{W}^{1}+ \T{W}^{2}\right) H^{-1} \left( W^{1} + W^{2}\right)\end{matrix}\right).
\]
By proposition \ref{lem:VRJPNu0} and lemma \ref{lem:Schur}, $H^-,H^+$ and $H^{\infty}$ are distributed according to $\tilde{\nu}_{n+2}^{W^-,0},\tilde{\nu}_{n+2}^{W^+,0}$ and $\tilde{\nu}_{n+1}^{W^{\infty},0}$ respectively. Let $G^-,G^+$ and $G^{\infty}$ be the inverse of $H^-,H^+$ and $H^{\infty}$ respectively. Let $G^{22,-},G^{22,+}$ and $G^{22,\infty}$ be defined by:
\[
\begin{aligned}
G^{22,-}:=&\left(\begin{matrix}G^-(n+1,n+1) & G^-(n+1,n+2) \\ G^-(n+2,n+1) & G^-(n+2,n+2)\end{matrix}\right),\\ 
G^{22,+}:=&\left(\begin{matrix}G^+(n+1,n+1) & G^+(n+1,n+2) \\ G^+(n+2,n+1) & G^+(n+2,n+2)\end{matrix}\right) \text{ and }\\
G^{22,\infty}:=& \left(G^{\infty}(n+1,n+1)\right).
\end{aligned}
\]
For any vector $X^2\in[0,\infty)^{n+2}$, by lemma \ref{lem:downto2} there exists three non-negative random variables $C(X^1,X^2),\alpha_1(X^2)$ and $\alpha_2(X^2)$ that only depend on $H,W^1$ and $W^2$ such that:
\[
\T{X^1} G^- X^2 =  C(X^1,X^2) + \left(\begin{matrix} \alpha_1(X^1) & \alpha_2(X^1) \end{matrix}\right) 
G^{22,-}
\left(\begin{matrix} \alpha_1(X^2) \\ \alpha_2(X^2) \end{matrix}\right),
\]
\[
\T{X^1} G^+ X^2 =  C(X^1,X^2) + \left(\begin{matrix} \alpha_1(X^1) & \alpha_2(X^1) \end{matrix}\right) 
G^{22,+}
\left(\begin{matrix} \alpha_1(X^2) \\ \alpha_2(X^2) \end{matrix}\right),
\]
and
\[
\T{\overline{X}^1} G^{\infty} \overline{X}^2 = 
C(X^1,X^2) + (\alpha_1(X^1) +\alpha_2(X^1)) G^{22,\infty} (\alpha_1(X^2) +\alpha_2(X^2)).
\]
Let $\alpha(X^2):=\alpha_1(X^2)+\alpha_2(X^2)$ and let $\theta\in[-1,1]$ be defined by:
\[
\theta:=\left\{\begin{matrix}\frac{\alpha_1(X^2)}{\alpha_1(X^2)+\alpha_2(X^2)} &\text{ if }\alpha_1(X^2)+\alpha_2(X^2)\not=0 \\ 0 &\text{ otherwise} \end{matrix}\right. .
\]
We have:
\[
\T{X^1} G^- X^2 =  C(X^1,X^2) + \alpha(X^1)\alpha(X^2)\left(\begin{matrix} \lambda & 1-\lambda \end{matrix}\right) 
G^{22,-}
\left(\begin{matrix} \theta \\ 1-\theta \end{matrix}\right),
\]
\[
\T{X^1} G^+ X^2 =  C(X^1,X^2) + \alpha(X^1)\alpha(X^2)\left(\begin{matrix} \lambda & 1-\lambda \end{matrix}\right) 
G^{22,+}
\left(\begin{matrix} \theta \\ 1-\theta \end{matrix}\right),
\]
and
\[
\T{\overline{X}^1} G^{\infty} \overline{X}^2 = 
C(X^1,X^2) + \alpha(X^1)\alpha(X^2)G^{22,\infty}.
\]
By lemma \ref{lem:martingale} and by definition of $U^-$ and $U^+$, we have:
\[
\E\left(\T{X}^1 G^- X^2|H^{+}\right)=\T{X}^1 G^+ X^2,
\]
and 
\[
\T{X}^1 G^- X^1 =C(X^1,X^1) \left(\alpha_1(X^1)+\alpha_2(X^1)\right)^2\frac{1}{\gamma} = \T{X}^1 G^+ X^1 = \T{\overline{X^1}} G^{\infty} \overline{X^1}.
\]
By lemmas \ref{lem:martingale} and \ref{lem:tiltedgaussianesp}, we have:
\[
\E\left(\T{X}^1 G^+ X^2|H^{\infty}\right)= \T{\overline{X^1}} G^{\infty} \overline{X^2}.
\] 
\end{proof}

\begin{proof}[proof of theorem \ref{maincor}]
Set an integer $i\in[\![1,n]\!]$. We will only show this result when $W^-$ and $W^+$ differ by only two symmetric coefficients (i.e one edge): $(k,l)$ and $(l,k)$. We can assume that $W^-(n-1,n)<W^+(n-1,n)$ because of the symmetries of the family of laws $\tilde{\nu}_n^{W,0}$. For any $j_1,j_2\in[\![1,n]\!]$, $j_1$ and $j_2$ are $W^-$-connected. This means that by the main theorem, there exists two matrices $H^-$ and $H^+$ distributed according to $\tilde{\nu}_n^{W^-,0}$ and $\tilde{\nu}_n^{W^+,0}$ respectively, with inverse $G^-$ and $G^+$ respectively and such that:
\begin{itemize}
\item $G^-(i,i)=G^+(i,i)$ almost surely,
\item $\forall X\in[0,\infty)^n,\ \E\left(\sum\limits_{j=1}^n X_j G^-(i,j)|H^+\right) = \sum\limits_{j=1}^n X_j G^+(i,j)$.
\end{itemize}
This means that for any convex function $f$ and any vector $X\in[0,\infty)^n$:
\[
\E\left(f\left(\frac{\sum\limits_{j=1}^n X_j G^-(i,j)}{G^-(i,i)}\right)\right)\geq \E\left(f\left(\frac{\sum\limits_{j=1}^n X_j G^+(i,j)}{G^+(i,i)}\right)\right).
\]
\end{proof}

\section{Proofs of theorems \ref{theo:TransCroiss},\ref{theo2} and \ref{theo3}}

\subsection{Proof of theorem \ref{theo:TransCroiss}}

\begin{proof}
Let $d_{\mathcal{G}}(\cdot ,\cdot)$ be the graph distance on $\mathcal{G}$. Let $\mathcal{G}_n$ be the graph obtained by fusing together all the vertices at a distance $n$ or more from $0$. This means that $\mathcal{G}_n=(V_n,E_n)$, with:
\[
\begin{aligned}
V_n=&\{x\in V, d_{\mathcal{G}}(0,x)<n\}\cup\{\delta_n\} \text{ and,}\\ 
E_n=&\left\{ \{x,y\}\in E, (x,y)\in\! V_n^2 \right\} \cup \left\{ \{x,\delta_n\}, d_{\mathcal{G}}(0,x)=n-1, \exists y \in\! V\backslash V_n,  d_{\mathcal{G}}(x,y)=1  \right\}.
\end{aligned}
\]
Let $|V_n|$ be the number of vertices in $V_n$. Let $W_n^-\in M_{|V_n|}(\R)$ and $W_n^+ \in M_{|V_n|}(\R)$ be the symmetric matrices defined by:
\begin{itemize}
\item for any $x,y\in V_n$ such that $\{x,y\}\not\in E_n$, $W_n^-(x,y)=W_n^+(x,y)=0$,
\item for any $x,y\in V_n \backslash \{ \delta \}$, $W_n^-(x,x)=W^-_{\{x,x\}}$ and $W_n^+(x,x)=W_{\{x,x\}}^+$,
\item for any $x \in V_n \backslash \{ \delta \}$, $W_n^-(x,\delta_n)=W_n^-(\delta_n,x)=\sum\limits_{y\in V,\{x,y\}\in E}W^-_{\{x,y\}}1_{y\not\in V_n}$
\item for any $x \in V_n \backslash \{ \delta \}$, $W_n^+(x,\delta_n)=W_n^+(\delta_n,x)=\sum\limits_{y\in V,\{x,y\}\in E}W^+_{\{x,y\}}1_{y\not\in V_n}$ 
\end{itemize}
This means that for any $x,y\in V_n$, $W^-_n(x,y)\leq W^+_n(x,y)$. Let $H_n^-$ and $H_n^+$ be two random matrices distributed according to $\tilde{\nu}_{|V_n|}^{W_n^-,0}$ and $\tilde{\nu}_{|V_n|}^{W_n^+,0}$ respectively. Let $G_n^-$ and $G_n^+$ be the inverse of $H_n^-$ and $H_n^+$ respectively. By Theorem 1 of $\cite{VRJPNu}$, there exists two non-negative random variables $\psi^-(0)$ and $\psi^+(0)$ such that:
\[
\begin{aligned}
&\frac{G_n^-(0,\delta_n)}{G_n^-(\delta_n,\delta_n)}\xrightarrow[n\rightarrow \infty]{} \psi^-(0) \text{ in law, and}\\
&\frac{G_n^+(0,\delta_n)}{G_n^+(\delta_n,\delta_n)}\xrightarrow[n\rightarrow \infty]{} \psi^+(0) \text{ in law}.
\end{aligned}
\]
Furthermore, still by theorem 1 of $\cite{VRJPNu}$, we have:
\[
\begin{aligned}
&\Prob\left(\text{The VRJP with initial weights } w^- \text{ is recurrent}\right)=\Prob\left(\psi^-(0)=0\right),\\
&\Prob\left(\text{The VRJP with initial weights } w^+ \text{ is recurrent}\right)=\Prob\left(\psi^+(0)=0\right).
\end{aligned}
\]
Let $f:[0,\infty)\mapsto \R$ be a continuous, bounded, convex function. By theorem \ref{maincor}, we have, for any $n\geq 1$: 
\[
\E\left(f\left(\frac{G_n^-(0,\delta_n)}{G_n^-(\delta_n,\delta_n)}\right)\right)\geq \E\left(f\left(\frac{G_n^+(0,\delta_n)}{G_n^+(\delta_n,\delta_n)}\right)\right).
\]
This means that $\E\left(f(\psi^-(0))\right)\geq \E\left(f(\psi^+(0))\right)$. For any $n\geq 1$, let $f_n:[0,\infty)\mapsto \R$ be the function defined by:
\[
f_n(x)=\left\{ \begin{matrix}
1-nx \text{ if } 0\leq x\leq \frac{1}{n} \\
0 \text{ if } x>\frac{1}{n}.
\end{matrix} \right.
\]
For any $n\geq 1$, the function $f_n$ is continuous, bounded and convex, so $\E\left(f_n(\psi^-(0))\right)\geq \E\left(f_n(\psi^+(0))\right)$. We notice that
\[
\begin{aligned}
&\E\left(f_n(\psi^-(0))\right)\xrightarrow[n\rightarrow \infty]{} \Prob\left(\psi^-(0)=0\right), \text{ and}\\
&\E\left(f_n(\psi^+(0))\right)\xrightarrow[n\rightarrow \infty]{} \Prob\left(\psi^+(0)=0\right).
\end{aligned}
\] 
This means that $\Prob\left(\psi^-(0)=0\right)\geq \Prob\left(\psi^+(0)=0\right)$ and therefore the probability that the VRJP with initial weights $w^-$ is recurrent is greater than the probability that the VRJP with initial weights $w^+ $ is recurrent.
\end{proof}

\subsection{Proof of theorem \ref{theo2}}

\begin{proof}
Set a dimension $d\geq 3$. By proposition 3 of \cite{VRJPNu}, for any $w\in (0,\infty)$, the VRJP on $\Z^d$ is either almost surely recurrent or almost surely transient. Furthermore, by theorem \ref{maincor}, the probability that the VRJP is recurrent is non-increasing in the initial weight. Therefore, there exists $w_d\in [0,\infty]$ such that the VRJP on $Z^d$ with initial weight $w\in (0,\infty)$ is recurrent if $w<w_d$ and transient if $w>w_d$. Since the VRJP is recurrent in dimension 3 for small enough weights (corollary 3 of \cite{ERRWVRJP}), $w_d\not=0$ and since it is transient for large enough weights (lemma 9 of \cite{VRJPNu}), $w_d\not = \infty$.
\end{proof}

\subsection{Proof of theorem \ref{theo3}}

\begin{proof}
Set a dimension $d\geq 3$. Let $E^d$ be the set of vertices in $\Z^d$. Set $0<a^-<a^+$. Let $(W^-_e)_{e\in E}$ be iid random Gamma variables with parameter $a^-$ and let $(W^{\prime}_e)_{e\in E}$ be iid random Gamma variables with parameter $a^+-a^-$. By theorem 1 of \cite{ERRWVRJP}, the ERRW on $\Z^d$ with initial weight $a^-\in (0,\infty)$ is a mixture of VRJP on $\Z^d$ where the initial weights are $(W^-_e)_{e\in E}$ and the ERRW on $\Z^d$ with initial weight $a^+\in (0,\infty)$ is a mixture of VRJP on $\Z^d$ where the initial weights are $(W^-_e + W^{\prime}_e)_{e\in E}$. Now, by theorem \ref{maincor}, the VRJP with initial weights $(W^-_e)_{e\in E}$ has a higher probability of being recurrent than the VRJP with initial weights $(W^-_e + W^{\prime}_e)_{e\in E}$. Therefore the probability that the ERRW with constant weight equal to $a$ is recurrent is non-increasing in $a$. By proposition 5 of $\cite{VRJPNu}$, the ERRW with initial weight $a$ is either almost surely transient or almost surely recurrent. Therefore, there exists $a_d\in [0,\infty]$ such that the ERRW on $\Z^d$ with initial weight $a\in (0,\infty)$ is recurrent if $a<a_d$ and transient if $a>a_d$. Since the ERRW is recurrent in dimension 3 for small enough weights, $a_d\not=0$ and since it is transient for large enough weights, $a_d\not = \infty$.
\end{proof}

\section{Proof of theorem \ref{theo:CondEff}}

\subsection{Preliminaries}

\begin{defin}
Let $\mathcal{G}=(V,E)$ be a finite graph and $(W_e)_{e\in E}$ be positive weights. Let $H_{\beta}$ be the random matrix distributed according to $\tilde{\nu}_n^{W,0}$ and $G_{\beta}$ its inverse. Let $x,y\in V$ be two distinct vertices of $\mathcal{G}$. The effective weight between $x$ and $y$, $w^{\text{eff}}_{x,y}$, is the random variable defined by:
\[
w^{\text{eff}}_{x,y}:= \frac{G_{\beta}(x,y)}{G_{\beta}(x,x)G_{\beta}(y,y)-G_{\beta}(x,y)^2}.
\]
\end{defin}

\begin{rmq}\label{rmq}
Let $\mathcal{G}=(V,E)$ be a finite graph and $(W_e)_{e\in E}$ be positive weights. Let $(\beta_i)_{i\in V}$ be random variables distributed according to $\nu_n^{W,0}$, $H_{\beta}$ the corresponding matrix (distributed according to $\tilde{\nu}_n^{W,0}$) and $G_{\beta}$ its inverse. Let $x,y\in V$ be two distinct vertices of $\mathcal{G}$ and $w^{\text{eff}}$ the effective weight between $x$ and $y$. Let $V_1:=\{x,y\}$ and $V_2:=V\backslash \{x,y\}$ be two subsets of $V$. The corresponding decomposition of $H_{\beta}$ is given by:
\[
H_{\beta}:=\left(\begin{matrix}H_{\beta}^{V_1} & -W^{V_1,V_2} \\ -\T{W}^{V_1,V_2} & H_{\beta}^{V_2}\end{matrix}\right).
\]
By lemma \ref{lem:Schur}, 
\begin{equation}
W^{\text{eff}}= W_{x,y}+\left(\T{W}^{V_1,V_2}\left(H_{\beta}^{V_2}\right)^{-1} W^{V_1,V_2}\right)(x,y).
\end{equation}
Furthermore, by lemmas \ref{lem:VRJPNu0} and \ref{lem:Schur}, the law of $\frac{G_{\beta}(x,y)}{G_{\beta}(y,y)}$ knowing the $\beta$-field on $V_2$ is the same as the law of $\frac{G_{\beta}(z_1,z_2)}{G_{\beta}(z_2,z_2)}$ on a two-vertices graph $\{z_1,z_2\}$ where $W_{z_1,z_2}=w^{\text{eff}}$.\\
\end{rmq}

\begin{lem}\label{lem:conduc}
Let $\mathcal{G}=(V,E)$ be a finite graph and $x_0,\delta \in V$ two distinct vertices. Let $(c_e)_{e\in E}$ be a family of random (not necessarily independent) positive conductances. Let $c_{\text{eff}}$ be the (random) effective conductance between $x_0$ and $\delta$ for the electrical network with initial conductances $(c_e)_{e\in E}$. Let $\overline{c}_{\text{eff}}$ be the equivalent conductance between $x_0$ and $\delta$ if we set conductances $(\overline{c}_e)_{e\in E}$ defined by $\overline{c}_e:=\E\left(c_e\right)$ on $\mathcal{G}$. We have the following inequality:
\[
\E(c_{\text{eff}}) \leq \overline{c}_{\text{eff}}.
\]
\end{lem}
\begin{proof}
Let $(V_x)_{x\in V}$ be the (random) potential with $V_{x_0}=1$ and $V_{\delta}=0$ that minimizes the energy:
\[
\mathcal{E}:=\frac{1}{2}\sum\limits_{\{x,y\}\in E} c_e(V_x-V_y)^2.
\]
This potential is harmonic on $V\backslash \{x_0,\delta\}$ by the Dirichlet principle and therefore $(V_x-V_y)_{(x,y)\in E}$ is the flow that minimizes the energy and we get: 
\[
\mathcal{E}:=\frac{1}{2} c_{\text{eff}}.
\]
Now let $(\overline{V}_x)_{x\in V}$ be the potential with $\overline{V}_{x_0}=1$ and $\overline{V}_{\delta}=0$ that minimizes the energy:
\[
\overline{\mathcal{E}}:=\frac{1}{2}\sum\limits_{\{x,y\}\in E} \overline{c}_e(\overline{V}_x-\overline{V}_y)^2.
\]
We have:
\[
\overline{\mathcal{E}}:=\frac{1}{2} \overline{c}_{\text{eff}}.
\]
Now since $V$ minimizes $\mathcal{E}$, we have:
\[
\mathcal{E}\leq \frac{1}{2}\sum\limits_{\{x,y\}\in E} c_e(\overline{V}_x-\overline{V}_y)^2.
\]
Now, by taking the expectation we get:
\[
\E\left(\mathcal{E}\right)\leq \frac{1}{2}\sum\limits_{\{x,y\}\in E} \E\left(c_e\right)(\overline{V}_x-\overline{V}_y)^2.
\]
Therefore:
\[
\frac{1}{2} \E\left(c_{\text{eff}}\right) \leq \frac{1}{2}\sum\limits_{\{x,y\}\in E_{n+1}} \E\left(c_e\right)(\overline{V}_x-\overline{V}_y)^2.
\]
Then we get:
\[
\frac{1}{2} \E\left(c_{\text{eff}}\right) \leq \frac{1}{2} \overline{c}_{\text{eff}}.
\]
And therefore:
\[
\E\left(c_{\text{eff}}\right) \leq \overline{c}_{\text{eff}}.
\]
\end{proof}

\begin{prop}\label{lem:EffWeight}
Let $\mathcal{G}=(V,E)$ be a finite graph and $x_0,\delta \in V$ two distinct vertices. Let $(W_e)_{e}$ be a family of random (not necessarily independent) positive weights. Let $w^{\text{eff}}$ be the (random) effective weight between $x_0$ and $\delta$ for the VRJP with initial weights $(W_e)_{e\in E}$. Let $c^{\text{eff}}$ be the effective conductance between $x_0$ and $\delta$ if we set conductances $(c_e)_{e\in E}$ defined by $c_e:=\E\left(W_e\right)$ on $\mathcal{G}$. We have the following inequality:
\[
\E\left(w_{\text{eff}}\right)\leq c_{\text{eff}}.
\]
\end{prop}
\begin{proof}
We will show the result by induction on the number of vertices of the graph. If the graph has two vertices $\{x_0,\delta\}$ (and therefore only one edge) the result is obvious.\\
Now we assume that the result is true for all graphs with $n$ vertices or less, we will show it for any graph with $n+1$ vertices.\\
Let $\mathcal{G}_{n+1}=(V_{n+1},E_{n+1})$ be a finite graph with exactly $n+1$ vertices, including $x_0$ and $\delta$. Let $(W^{n+1}_e)_{e\in E_{n+1}}$ be random weights on $E_{n+1}$. Let $H_{\beta}$ be a random matrix distributed according to $\tilde{\nu}_n^{W,0}$. Let $w^{\text{eff}}_{n+1}$ be the (random) effective weight between $x_0$ and $\delta$. Let $(c^{n+1}_e)_{e\in E_{n+1}}$ be deterministic conductances defined by $c^{n+1}_e=\E\left(W^{n+1}_e\right)$. We define two effective conductances between $x_0$ and $\delta$ on $\mathcal{G}_{n+1}$: one for random conductances $W^{n+1}$ ($\overline{c}_{n+1}^{\text{eff}}$) and the other for deterministic conductances $(c_e)_{e\in E}$ ($c_{n+1}^{\text{eff}}$). By lemma \ref{lem:conduc}: 
\begin{equation}
\E\left(\overline{c}_{n+1}^{\text{eff}}\right)\leq\E\left(\overline{c}_{n+1}^{\text{eff}}\right).
\end{equation} 
Now, let $y\in V_{n+1}$ be a vertex that is neither $x_0$ nor $\delta$.
Let $\mathcal{G}^y_{n}=(V^y_{n},E^y_{n})$ be the complete graph with $n$ elements with $V^y_{n}=V_{n+1}\backslash \{y\}$. We can decompose $V_{n+1}$ in $V^y_n$ and $\{y\}$, the corresponding decomposition of $H_{\beta}$ is given by:
\[
H_{\beta}:=\left(\begin{matrix} H_{\beta}^{V_n} & -W^{V_1,y} \\ - \T{W}^{V_1,y} & 2\beta_y\end{matrix}\right).
\]
By lemma \ref{lem:Schur}, $w^{\text{eff}}_{n+1}$ knowing $H_{\beta}$ is equal to the effective weight $w^{\text{eff}}_{n}$ on the graph $\mathcal{G}^y_{n}$ for weights and the $\beta$-field given by the matrix $H_{\beta}^{V^y_n} - \frac{1}{2\beta_y} W^{V^y_1,y} \T{W}^{V^y_1,y}$. This matrix, knowing $\beta_y$ and $W^{n+1}$ is distributed according to $\nu_n^{W^{\prime}}$ with $W^{\prime}_{x_1,x_2}=W^{n+1}_{x_1,x_2}+\frac{W^{n+1}_{x_1,y}W^{n+1}_{y,x_2}}{2\beta_y}$. By \ref{lem:VRJPNu0}, if $K_y:=\sum\limits_{x,\{x,y\}\in E_{n+1}} W^{n+1}_{y,x}$ the expectation of $\frac{1}{2\beta_y}$, knowing $W^{n+1}$ is given by:
\[
\begin{aligned}
\E\left(\frac{1}{2\beta_y}\right)
=& \int\limits_{b=0}^{\infty}\frac{1}{2b}\sqrt{\frac{2}{\pi}}\frac{1}{\sqrt{2b}}\exp\left(-\frac{1}{2}\left(2b + \frac{K_y^2}{2b}-2K_y\right)\right)\dd b\\
=& \int\limits_{b=0}^{\infty} \frac{b}{2}\sqrt{\frac{2}{\pi}}\sqrt{\frac{b}{2}}\exp\left(-\frac{1}{2}\left(\frac{2}{b} + \frac{K_y^2 b}{2}-2K_y\right)\right) \frac{1}{b^2}\dd b \\
=& \frac{1}{2}\int\limits_{b=0}^{\infty}\frac{1}{\sqrt{2b}}\sqrt{\frac{2}{\pi}}\exp\left(-\frac{1}{2}\left(\frac{4}{2b} + \frac{K_y^2}{4}2b-2K_y\right)\right) \dd b\\
=& \frac{1}{K_y}\int\limits_{b=0}^{\infty}\frac{1}{\sqrt{2b}}\sqrt{\frac{2}{\pi}}\exp\left(-\frac{1}{2}\left(\frac{K_y^2}{2b} + 2b-2K_y\right)\right) \dd b\\
=& \frac{1}{K_y} \text{by definition of }\nu_1^{0,K_y}.
\end{aligned}
\]
Therefore for any $x_1,x_2\in V^y_n$:
\[
\E\left(W^{\prime}_{x_1,x_2}|W^{n+1}\right)=W^{n+1}_{x_1,x_2}+\frac{W^{n+1}_{x_1,y}W^{n+1}_{y,x_2}}{\sum\limits_x W^{n+1}_{y,x}}.
\]
Similarly the effective conductance $\overline{c}^{\text{eff}}_{n+1}$ between $x_0$ and $\delta$ on $\mathcal{G}_{n+1}$ with conductances $W^{n+1}$ is equal to the effective conductance $\overline{c}^{\text{eff}}_{n}$ between $x_0$ and $\delta$ on $\mathcal{G}^y_{n}$ with conductances $\overline{c}^{\prime}_{x_1,x_2}:=W^{n+1}_{x_1,x_2}+\frac{W^{n+1}_{x_1,y}W^{n+1}_{y,x_2}}{\sum\limits_x W_{y,x}}$.
This means that, for any $e\in E^y_{n}$:
\[
\E\left(W_e^{\prime}|W^{n+1}\right)=\overline{c}^{\prime}_e,
\] 
so by the induction property:
\[
\E\left(w^{\text{eff}}_{n}\right)\leq\E\left(\overline{c}^{\text{eff}}_{n}\right),
\]
which implies that
\[
\E\left(w^{\text{eff}}_{n+1}\right)\leq\E\left(c^{\text{eff}}_{n+1}\right).
\]
\end{proof}

\subsection{proof of theorem \ref{theo:CondEff}}

\begin{proof}
Once we can compare the effective weight for the VRJP to effective conductance for an electrical network, the proof is quite straightforward. Let $\tilde{W}_e$ be weights and let $c_e:=\E(\tilde{W}_e)$ be conductances. For any $n>0$ we define $\overline{S}_n$ the vertices of $V$ at distance $n$ or more of $x_0$. Then $\mathcal{G}_n,\tilde{W}^n$ (with $\mathcal{G}_n:=(V_n,E_n)$) is the quotient of the weighted graph $\mathcal{G},\tilde{W}$ by $\overline{S}_n$ and $\delta_n$ is the point obtained by fusing all points of $\overline{S}_n$ into one. For any $n$, let $H_n$ be distributed according to $\tilde{\nu}_{|V_n|}^{\tilde{W}^n,0}$ and let $G_n$ be its inverse. By Theorem 1 of $\cite{VRJPNu}$, to show that the VRJP with initial weights $\tilde{W}_e$ is recurrent, we only need to show that $\frac{G_n(x_0,\delta_n)}{G_n(\delta_n,\delta_n)}$. By remark \ref{rmq}, the law of $\frac{G_n(x_0,\delta_n)}{G_n(\delta_n,\delta_n)}$ is entirely determined by the law of the effective weight. Since the effective conductive converges to $0$, the effective weights converges to $0$ in probability by lemma \ref{lem:EffWeight}. Then, by remark \ref{rmq}, the law of 
$\frac{G_n(x_0,\delta_n)}{G_n(\delta_n,\delta_n)}$ knowing the effective weight is the same as if the graph had only two points: $x_0$ and $\delta$ with a weight equal to the effective weight between them. Now let $(\beta_1,\beta_2)$ be distributed according to $\nu_2^{w^{\text{eff}},0}$, the law of $\frac{G_n(x_0,\delta_n)}{G_n(\delta_n,\delta_n)}$ is the same as the law of 
\[
\frac{\frac{w^{\text{eff}}}{4\beta_1\beta_2-(w^{\text{eff})^2}}}{\frac{2\beta_1}{4\beta_1\beta_2-(w^{\text{eff})^2}}}
=\frac{w^{\text{eff}}}{2\beta_1}.
\] 
By taking $\lambda=1$ in lemma \ref{lem:coupling}, we get that 
\[
\frac{w^{\text{eff}}}{2\beta_1}=\frac{w^{\text{eff}}}{W^{\text{eff}\frac{1}{z}}}=Z,
\]
where the law of $Z$ (knowing $w^{\text{eff}}$) is given by:
\[
\sqrt{\frac{w^{\text{eff}}}{2\pi}}\frac{1}{z\sqrt{z}}\exp\left(-\frac{w^{\text{eff}}}{2}\left(\sqrt{z}-\frac{1}{\sqrt{z}}\right)^2\right) 1_{z>0} \dd z.
\]
If $w^{\text{eff}}$ goes to $0$ then $Z$ converges to $0$ in probability and therefore $\frac{G_n(x_0,\delta_n)}{G_n(\delta_n,\delta_n)}$ converges to $0$ in probability and we get the result we want.
\end{proof}

\section{Acknowledgement}

I would like to thank Christophe Sabot, my Phd advisor, for suggesting working in this direction and helpful discussions. I would also like to thank Bruno Schapira for suggesting that our main result could be used to show recurrence of the VRJP on recurrent graphs.

\bibliographystyle{abbrv}
\bibliography{biblio}

\end{document}